\documentclass[table,xcdraw]{article}
\usepackage[utf8]{inputenc}
\usepackage[margin=1in]{geometry}
\usepackage{amsfonts}
\usepackage{amsmath}
\usepackage[shortlabels]{enumitem}
\usepackage{amsthm}
\usepackage{amssymb}
\usepackage[style=alphabetic]{biblatex}
\renewbibmacro{in:}{}
\usepackage{xr}
\usepackage{fancyhdr}
\usepackage{subcaption}
\usepackage{longtable}
\usepackage{amsmath}
\usepackage{mathtools}
\usepackage{blkarray}
\usepackage{bm}
\usepackage{hyperref}
\hypersetup{
    colorlinks=true,
    linkcolor=blue,
    filecolor=magenta,
    urlcolor=cyan,
    citecolor=blue
}
\bibliography{sources.bib}

\usepackage{url}
\usepackage{float}
\usepackage{graphicx}
\usepackage{circuitikz}
\usepackage{quiver}
\usepackage{mathrsfs}
\usepackage{abstract}
\title{The Galois Structure of the Spaces of polydifferentials on the Drinfeld Curve}
\author{Bernhard K\"ock \and Denver-James Marchment}
\date{}

\newtheorem{thm}{Theorem}[section]

\newtheorem*{thm*}{Theorem}
\newtheorem{cor}[thm]{Corollary}

\newtheorem{lem}[thm]{Lemma}
\newtheorem{prop}[thm]{Proposition}

\theoremstyle{definition}
\newtheorem{defn}[thm]{Definition}

\newtheorem{rem}[thm]{Remark}

\newtheorem*{rem*}{Remark}
\newcommand{\Char}{\text{char}}

\newcommand{\thh}{\text{th}}

\newcommand{\Mod}[1]{\ \mathrm{mod}\ (#1)}
\newcommand{\Modwb}[1]{\ \mathrm{mod}\ #1}

\newcommand{\Rad}{\textrm{{Rad}}}
\newcommand{\Ram}{\textrm{\normalfont R}}
\newcommand{\Aut}{\text{Aut}}
\newcommand{\Ord}{\textrm{\normalfont ord}}
\newcommand{\Res}{\textrm{\normalfont Res}}
\newcommand{\Ind}{\textrm{\normalfont Ind}}

\newcommand{\Br}{\textrm{\normalfont Br}}

\newcommand{\FF}{\mathbb{F}}

\newcommand{\ind}{\textrm{{\normalfont Ind}}}
\DeclareMathOperator{\Span}{\normalfont Span}

\usepackage{mathtools}
\DeclarePairedDelimiter\ceil{\lceil}{\rceil}
\DeclarePairedDelimiter\floor{\lfloor}{\rfloor}
\usepackage{circuitikz}
\usepackage{float}
\usepackage{enumitem}
\usepackage{underoverlap}\usepackage[justification=centering]{caption}
\begin{document}

\maketitle

\begin{abstract}
    Let $C$ be a smooth projective curve over an algebraically closed field $\FF$ equipped with the action of a finite group $G$. When $p =\textrm{char}(\FF)$ divides the order of $G$, the long-standing problem of computing the induced representation of $G$ on the space $H^0(C,\Omega^{\otimes m}_C)$ of globally holomorphic polydifferentials remains unsolved in general. In this paper, we study the case of the group $G = \mathrm{SL}_2(\FF_q)$ (where $q$ is a power of~$p$) acting on the Drinfeld curve $C$ which is the projective plane curve given by the equation $XY^q-X^qY-Z^{q+1} = 0$. When $q = p$, we fully decompose $H^0(C,\Omega^{\otimes m}_C)$ as a direct sum of indecomposable $\FF[G]$-modules. For arbitrary $q$, we give a partial decomposition in terms of an explicit $\FF$-basis of $H^0(C,\Omega^{\otimes m}_C)$. Finally, in the appendix, we compute the $a$-number and $p$-rank of the Drinfeld curve.
\end{abstract}

\section{Introduction}
Given a smooth projective curve $C$ over an algebraically closed field ${\mathbb{F}}$ and a positive integer ${m \in \mathbb{Z}^+}$, one defines the space ${H^0(C,\Omega_{C}^{\otimes m})}$ whose elements are called globally holomorphic polydifferentials (or just differentials when ${m=1}$). This is a finite-dimensional vector space over the field $\mathbb{F}$, whose dimension is given by
\begin{equation}
    \label{dimension of globally holomorphic polydifferentials}
    \dim_{\mathbb{F}}H^0(C,\Omega_C^{\otimes m}) =
    \begin{dcases}
    g(C)&\text{if }m=1\\
    (2m-1)(g(C)-1)&\text{if }m>1
    \end{dcases}.
\end{equation}
If a finite group $G$ acts on the curve $C$, we obtain an induced action on this space. In other words, we obtain representations. The case of ${m=1}$ is given a special name: the \textit{canonical representation} associated with the action of $G$ on $C$.\\
In 1928, Hecke first posed the problem of writing the canonical representation as a direct sum of indecomposable modules (\cite{Hecke_1928}). When ${\Char(\mathbb{F})}$ does not divide the order of the group (i.e. when ${\mathbb{F}[G]}$ is semi-simple), this problem was solved by Chevally and Weil in 1934 (\cite{Chevalley_Weil1934}); however, the modular case (i.e. when ${\Char(\mathbb{F}) = p > 0}$ divides ${|G|}$) is not solved in full generality. By making various assumptions about ramification or the group itself, numerous authors have obtained results regarding the decomposition of ${H^0(C,\Omega_{C}^{\otimes m})}$. For instance, if $G$ is a cyclic $p$-group or is an elementary abelian group of order ${p^n}$, (\cite{Karanikolopoulos}) gives the decomposition using an explicit basis. The papers \cite{Kani_1986} and \cite{Nakajima_1986} give the decomposition of the canonical representation in the case that ${C\to C/G}$ is tamely ramified. This was then generalized to the case of weak ramification  (the simplest form of wild ramification) in \cite{Kock_2004}. Two papers relevant for our purposes are \cite{bleher_justfordifferentials} (which works with ${m=1}$) and \cite{bleher} (which works with ${m > 1}$). These articles work out the decomposition when ${\Char(\mathbb{F}) = p > 0}$ and ${G = U\rtimes T}$, where $U$ is a cyclic Sylow-${p}$ subgroup of $G$ and ${T}$ is a cyclic subgroup of $G$ with order coprime to $p$. This makes $G$ a so-called \textit{p-hypo-elementary group}.\\
In this paper, we will be looking at the case of the Drinfeld curve, which holds historical importance as a prototypical example in Deligne-Lusztig theory. We let ${q = p^r}$ be a prime power and continue to denote by ${\mathbb{F}}$ an algebraically closed field of characteristic $p$. The Drinfeld curve, which we shall denote from now on by $C$, is the smooth projective curve over the field ${\mathbb{F}}$ defined as the zero locus of the equation ${XY^q - X^qY - Z^{q+1}}$ in the projective plane ${\mathbb{P}^2(\mathbb{F})}$. Let ${\mathbb{F}_q}$ denote a finite field of order $q$. Then the group ${G = SL_2(\mathbb{F}_q)}$ acts on the Drinfeld curve by
\begin{equation}
    \label{drinfeld curve action}
    \left(
        \begin{matrix}
            \alpha&\beta\\
            \gamma&\delta
        \end{matrix}
    \right)\cdot[X:Y:Z]
    =
    [\alpha X + \beta Y:\gamma X + \delta Y:Z].
\end{equation}
The canonical representation of ${G}$ on $C$ has been written as a direct sum of indecomposable modules in \cite{lucas}. Interestingly, the canonical representation is semi-simple if and only if ${q=p}$ (\cite[\S 2, Thm. 2.1]{lucas}).\\
The purpose of this report is to study the case of ${m\geq 2}$ for the Drinfeld curve. For general ${q}$, we produce a basis for ${H^0(C,\Omega_C^{\otimes m})}$ (Proposition \ref{basis for holomorphic polydifferentials}), along with a partial ${\mathbb{F}[G]}$-module decomposition (Proposition \ref{partial decomposition of polydifferentials}). From sections 3-5, we restrict to the case ${q=p}$. Under this assumption, we give the full decomposition of ${H^0(C,\Omega_C^{\otimes m})}$ as a direct sum of indecomposable modules. In order to do this, we first compute the decomposition for the subgroup $B$ of upper triangular matrices in $G$. When ${q=p}$, the indecomposable ${\mathbb{F}[B]}$ modules are denoted by ${U_{a,b}}$, with ${0\leq a\leq p-2}$ controlling the socle and ${1\leq b\leq p}$ its dimension (see Proposition \ref{parameterisation of F[B] modules}). By using results from \cite{bleher}, we obtain explicit formulae for the multiplicities ${n_{a,b}}$ in the decomposition
$$
{
    \Res^G_B(H^0(C,\Omega_{C}^{\otimes m}))
    \cong
    \bigoplus_{b=1}^{p}\bigoplus_{a=0}^{p-2}U_{a,b}^{\oplus n_{a,b}}
}
$$
(see Theorem \ref{decomposition as FB module, q=p}). From here, we use the fact that $G$ and $B$ are in Green correspondence (Theorem \ref{green correspondence}) to give a parameterisation of the non-projective indecomposable ${\mathbb{F}[G]}$ modules, and obtain the stable (i.e. non-projective) decomposition of ${H^0(C,\Omega_C^{\otimes m})}$. It follows we can write
$$
{
    H^0(C,\Omega_{C}^{\otimes m})
    \cong
    \bigoplus_{b=1}^{p-1}\bigoplus_{a=0}^{p-2}V_{a,b}^{\oplus n_{a,b}}
    \oplus
    \bigoplus_{i=1}^{p}P_{V_i}^{\oplus n_i}\text{ for some }n_i \in \mathbb{N}
}
$$
where ${V_{a,b}}$ denotes the Green correspondent of ${U_{a,b}}$ and ${P_{V_i}}$ denotes the projective cover of the simple ${\mathbb{F}[G]}$ module ${V_i}$ of dimension $i$ (Propositions \ref{simple F[G] modules}, \ref{projective indecomposable F[G] modules} and Theorem \ref{green correspondence}). After computing the composition factors of ${H^0(C,\Omega_C^{\otimes m})}$ as an ${\mathbb{F}[G]}$ module (Theorem \ref{composition factors of H^0}), we finally obtain explicit formula for the above ${n_i}$ (see Theorem \ref{lifting decomposition} and Remark \ref{F[G] module decomposition of H^0}) up to some constants which depend on the choice of ${p}$ (see Theorem \ref{c_{a,b,t} coefficients}). This gives us the full decomposition of ${H^0(C,\Omega_C^{\otimes m})}$ into a direct sum of indecomposable ${\mathbb{F}[G]}$ modules.\\
In the appendix, using the explicit basis of ${H^0(C,\Omega_C)}$ exhibited in Proposition \ref{basis for holomorphic polydifferentials}, we compute the $a$-number and $p$-rank of the Drinfeld curve (for arbitrary ${q = p^n}$). These are two non-negative integral invariants of the curve related to the $p$-torsion of its Jacobian. we show that the Drinfeld curve has a $p$-rank of $0$ (Proposition \ref{p-rank of drinfeld curve}), and that its $a$-number is given by ${\frac{q(q+p)(p-1)}{4p}}$ (Proposition \ref{a-number computation}).

\section{Notation}\label{notation-section}
We give here a list of core notations used throughout this paper. Throughout this table, we use ${\mathcal{G}}$ to denote a finite group acting on a smooth projective curve ${\mathcal{C}}$ whenever a general group action is under consideration.
\begin{longtable}{|l|l|}
\hline
Notation                         & Explanation                                                                                                                                                                                                                                                                             \\ \hline
\endfirsthead

\hline
Notation                         & Explanation                                                                                                                                                                                                                                                                             \\ \hline
\endhead

\hline
\endfoot

\hline
\endlastfoot

$p$                              & An odd prime.                                                                                                                                                                                                                                                                           \\ \hline
$q$                              & ${q = p^r}$ for some ${r \in \mathbb{N}}$, ${r\geq 1}$.                                                                                                                                                                                                                                 \\ \hline
${\mathbb{F}}$                   & An algebraically closed field of characteristic $p$.                                                                                                                                                                                                                                    \\ \hline
$C$                              & The Drinfeld Curve.                                                                                                                                                                                                                                                                     \\ \hline
${\mathbb{N}\ni m\geq 2}$        & The integer $m$ appearing in ${H^0(C,\Omega_C^{\otimes m})}$.                                                                                                                                                                                                                           \\ \hline
${\zeta}$                        & A primitive root modulo $p$.                                                                                                                                                                                                                                                            \\ \hline
$G$                              & The group ${G = SL_2(\mathbb{F}_q)}$.                                                                                                                                                                                                                                                   \\ \hline
${U}$                            & The subgroup of upper uni-triangular matrices of $G$. When ${q=p}$, ${U = \left\langle \left(\begin{matrix}1&1\\ 0&1\end{matrix}\right) \right\rangle\cong \mathbb{Z}/p}$.                                                                                                              \\ \hline
${T}$                            & The subgroup of diagonal matrices of $G$. When ${q=p}$, ${T = \left\langle \left(\begin{matrix}\zeta&0\\ 0&\zeta^{-1}\end{matrix}\right) \right\rangle\cong \mathbb{Z}/(p-1)}$.                                                                                                         \\ \hline
${B}$                            & The subgroup of upper triangular matrices of $G$. When ${q=p}$, ${U\rtimes T = B = N_G(U)}$.                                                                                                                                                                                            \\ \hline
${S_a}$                          & \begin{tabular}[c]{@{}l@{}}The simple ${\mathbb{F}[T]}$ modules. We also regard these as ${\mathbb{F}[B]}$ modules by inflation via the quotient\\ map ${B \twoheadrightarrow T}$, i.e. by letting $U$ act trivially. These are then the simple ${\mathbb{F}[B]}$ modules.\end{tabular} \\ \hline
${U_{a,b}}$                      & The indecomposable ${\mathbb{F}[B]}$ modules.                                                                                                                                                                                                                                           \\ \hline
${V_t}$                          & The simple ${\mathbb{F}[G]}$ modules when ${q=p}$.                                                                                                                                                                                                                                      \\ \hline
${P_S}$                          & The projective cover of the simple module ${S}$.                                                                                                                                                                                                                                        \\ \hline
${g(\mathcal{C})}$               & The genus of ${\mathcal{C}}$.                                                                                                                                                                                                                                                           \\ \hline
${H^0(C,\Omega_C^{\otimes m})}$  & The space of globally holomorphic polydifferentials of order ${m\geq 2}$ on $C$.

\\ \hline

${R_f}$                          & The ramification divisor associated to a morphism of curves ${f}$.                                                                                                                                                         \\ \hline

\\ \hline
${\mathcal{G}_{P,i}}$            & The ${i-\text{th}}$ lower ramification group at a point ${P \in \mathcal{C}}$ associated with the action of ${\mathcal{G}}$ on ${\mathcal{C}}$.                                                                                                                                         \\ \hline
${\mathfrak{m}_{\mathcal{C},P}}$ & The maximal ideal of the local ring ${\mathcal{O}_{\mathcal{C},P}}$ at ${P \in \mathcal{C}}$.                                                                                                                                                                                           \\ \hline
${\theta_P}$                     & \begin{tabular}[c]{@{}l@{}}Denotes the fundamental character at ${P \in \mathcal{C}}$. That is, the character associated with the action of\\ ${\mathcal{G}_P}$ on ${\mathfrak{m}_{\mathcal{C},P} / \mathfrak{m}_{\mathcal{C},P}^2}$.\end{tabular}                                      \\ \hline
${e_P}$                          & The ramification index at the point ${P \in \mathcal{C}}$ under the quotient map ${\mathcal{C} \to \mathcal{C} / \mathcal{G}}$.                                                                                                                                                         \\ \hline

\end{longtable}

\section{An explicit basis of ${H^0(C,\Omega_C^{\otimes m})}$}
We will compute an explicit basis for the holomorphic polydifferentials for the Drinfeld curve for arbitrary ${q=p^r}$ (${r \in \mathbb{N}}$, ${r\geq 1}$), and give a partial ${\mathbb{F}[G]}$-module decomposition for this basis. We use the notations given in \S \ref{notation-section}.
\begin{prop}
    \label{genus of C}
    The Drinfeld curve $C$ has genus ${g(C)=\frac{q(q-1)}{2}}$.
\end{prop}
\begin{proof}
    See \cite[\S 2.5.1]{bonnafe}.
\end{proof}
\begin{prop}
    \label{basis for holomorphic polydifferentials}
    Let $x$ and $y$ denote the rational functions ${X/Z}$ and ${Y/Z}$ on the Drinfeld curve respectively. For ${i,j \in \mathbb{Z}}$, define the meromorphic polydifferential
    \begin{equation}
        \label{omega_ij definition}
        \omega_{ij} := \frac{x^iy^j}{x^{mq}}dx^{\otimes m} \in \Omega_{\mathbb{F}(C)/\mathbb{F}}^{\otimes m}.
    \end{equation}
    For ${0\leq i,j\leq m(q-2),\ 0\leq i+j\leq m(q-2)}$, the polydifferential ${\omega_{ij}}$ is holomorphic, i.e. ${\omega_{ij} \in H^0(C,\Omega_C^{\otimes m})}$. Furthermore, if ${m=1}$,
    $$
    {
        \{\omega_{ij}: 0\leq i,j\leq q-2,\ i+j\leq q-2\}
    }
    $$
    is a basis for ${H^0(C,\Omega_{C})}$ over ${\mathbb{F}}$. Otherwise, if ${m\geq 2}$, then
    \begin{equation}
        \label{basis set for holomorphic polydifferentials}
        \mathcal{S} := \{\omega_{ij}: 0\leq j\leq q-1,\ 0\leq i\leq m(q-2)-j\}\cup \{\omega_{0j}: q\leq j\leq m(q-2)\}
    \end{equation}
    is a basis for ${H^0(C,\Omega_C^{\otimes m})}$ over ${\mathbb{F}}$.
\end{prop}
\begin{proof}
    The case of ${m=1}$ is simply \cite[\S 3, Prop. 3.2]{lucas}. We follow a similar argument for ${m\geq 2}$. We first show that ${\omega_{ij}}$ is holomorphic for ${i,j}$ in the given bounds. Suppose ${P = [x_0:y_0:z_0]}$ is a point on $C$.\\
    \\
    We first consider ${z_0 \neq 0}$. In this case, ${x - x_0/z_0}$ is a local parameter at $P$ by \cite[\S 3, Lem. 3.1]{lucas}. We have ${d(x-x_0/z_0) = dx}$, and so ${\omega_{ij}}$ is holomorphic at $P$ if ${x^iy^j/x^{mq}}$ is regular at $P$. Since ${x,y}$ satisfy the equation ${xy^q-x^qy-1 = 0}$, we have ${x\neq 0}$ at $P$, and hence ${x^iy^j/x^{mq}}$ is regular at $P$ for any ${i,j \in \mathbb{Z}}$.\\
    \\
    Next, we consider ${z_0 = 0,\ x_0 \neq 0}$. In this case, ${t = Z/X}$ is a local parameter at $P$ by \cite[\S 3 Lem. 3.1]{lucas}. Let ${s = Y/X}$. Then:
    $$
    {
        \omega_{ij} = \left(\frac{1}{t}\right)^{i-mq}\left(\frac{s}{t}\right)^jd\left(\frac{1}{t}\right)^{\otimes m} = (-1)^m t^{mq-i-j-2m}s^j dt^{\otimes m}.
    }
    $$
    This is holomorphic at ${P}$ if ${mq-i-j-2m\geq 0}$ and ${j\geq 0}$, which is equivalent to ${j\geq 0}$ and ${i\leq m(q-2)-j}$.\\
    \\
    Finally, we consider ${z_0=0,\ y_0 \neq 0}$. In this case, ${w = Z/Y}$ is a local parameter at $P$ by \cite[\S 3 Lem. 3.1]{lucas}. Let ${v = X/Y}$. Then:
    $$
    {
        \omega_{ij} = \left(\frac{v}{w}\right)^{i-mq}\left(\frac{1}{w}\right)^jdx^{\otimes m}.
    }
    $$
    Now we use the fact that ${dx = (x/y)^qdy}$ (which can be derived from differentiating ${xy^q - x^qy - 1 = 0}$):
    $$
    {
        \omega_{ij} = v^{i-mq}w^{mq-i-j}\left((x/y)^qdy\right)^{\otimes m} = v^{i-mq}w^{mq-i-j}\left(v^q d\left(\frac{1}{w}\right)\right)^{\otimes m} = (-1)^m v^iw^{mq-i-j-2m}dw^{\otimes m}.
    }
    $$
    This is holomorphic at $P$ if ${mq-i-j-2m \geq 0}$ and ${i\geq 0}$. To conclude, ${\omega_{ij}}$ is holomorphic if
    $$
    {
        \begin{array}{l}
             i\geq 0\\
             j\geq 0\\
             i + j \leq m(q-2)
        \end{array}
    }
    $$
    which is equivalent to the stated bounds.\\
    \\
    We now show that the elements in ${\mathcal{S}}$ are linearly independent. Assume
    $$
    {
        \sum_{0\leq j\leq q-1,\ 0\leq i\leq m(q-2)-j}c_{ij}\omega_{ij} + \sum_{q\leq j\leq m(q-2)}c_{0j}\omega_{0j} = 0
    }
    $$
    for some ${c_{ij} \in \mathbb{F}}$. Define the polynomials
    \begin{align*}
        F_1 := &\; \sum_{0\leq j\leq q-1,\ 0\leq i\leq m(q-2)-j}c_{ij}\mathcal{X}^i\mathcal{Y}^j&\in \mathbb{F}[\mathcal{X},\mathcal{Y}]\\
        F_2 := &\; \sum_{q\leq j\leq m(q-2)}c_{0j}\mathcal{Y}^j&\in \mathbb{F}[\mathcal{Y}]\\
        F := &\; F_1  + F_2&\in \mathbb{F}[\mathcal{X},\mathcal{Y}].
    \end{align*}
    Then ${F(x,y)/x^{mq}}$ is a rational function that vanishes everywhere on $C$. In particular, it vanishes on the affine section ${C_Z := \{[x:y:1] \in C\}}$. We have ${1/\mathcal{X}^{mq} \neq 0}$ on ${C_Z}$ and thus ${F(\mathcal{X},\mathcal{Y})}$ must vanish on all of ${C_Z}$. Hence ${F \in \textrm{Rad}(\mathcal{X}\mathcal{Y}^q - \mathcal{X}^q\mathcal{Y} - 1)}$ by Hilbert's Nullstellensatz.\\
    \\
    Since ${\mathcal{X}\mathcal{Y}^q - \mathcal{X}^q\mathcal{Y} - 1}$ is an irreducible polynomial in ${\mathbb{F}[\mathcal{X},\mathcal{Y}]}$ (see proof of \cite[\S 2.1, Prop. 2.1.1]{bonnafe}), we have ${\Rad(\mathcal{X}\mathcal{Y}^q-\mathcal{X}^q\mathcal{Y}-1) = (\mathcal{X}\mathcal{Y}^q-\mathcal{X}^q\mathcal{Y}-1)}$, and so ${F \in (\mathcal{X}\mathcal{Y}^q-\mathcal{X}^q\mathcal{Y}-1)}$, i.e. ${\mathcal{X}\mathcal{Y}^q-\mathcal{X}^q\mathcal{Y}-1}$ must divide ${F(\mathcal{X},\mathcal{Y})}$. We show this means ${F}$ must be $0$, i.e. all the coefficients are $0$, as to be shown.\\
    \\
    To this end, let ${\deg_\mathcal{Y}(-)}$ denote the degree of a polynomial when considered as a polynomial in ${(\mathbb{F}[\mathcal{X}])[\mathcal{Y}]}$. Suppose ${F\neq 0}$. Then ${F_2\neq 0}$, because ${F_2 = 0}$ would imply that ${F_1 \neq 0}$ and ${\mathcal{X}\mathcal{Y}^q-\mathcal{X}^q\mathcal{Y}-1}$ divides ${F_1}$. This cannot happen, since ${\deg_\mathcal{Y}(\mathcal{X}\mathcal{Y}^q-\mathcal{X}^q\mathcal{Y}-1) = q > q-1 \geq \deg_\mathcal{Y}(F_1)}$. Hence we have ${\deg_\mathcal{Y}(F) = \deg_\mathcal{Y}(F_2)}$.\\
    \\
    Now Write ${F = (\mathcal{X}\mathcal{Y}^q-\mathcal{X}^q\mathcal{Y}-1)h(\mathcal{X},\mathcal{Y})}$ for some ${h(\mathcal{X},\mathcal{Y})}$. By rearranging we obtain
    $$
    {
        h = \mathcal{X}(\mathcal{Y}^{q}h-\mathcal{X}^{q-1}h) - F_1 - F_2
    }
    $$
    and hence
    $$
    {
        h \equiv -\sum_{0\leq j\leq m(q-2)}c_{0j}\mathcal{Y}^{j} =: -F_3\Modwb{\mathcal{X}}
    }
    $$
    (${F_3}$ is the sum of ${F_2}$ and the ${\mathcal{Y}}$-terms of ${F_1}$). This implies ${\deg_y(h)\geq \deg_y(-F_3)}$. On the other hand, we have
    \begin{align*}
        &\deg_\mathcal{Y}(-F_3)&\\
        =&\deg_\mathcal{Y}(F_2)&\\
        =&\deg_\mathcal{Y}(F)&\\
        =&\deg_\mathcal{Y}((\mathcal{X}\mathcal{Y}^{q} - \mathcal{X}^q\mathcal{Y} - 1)h)&\\
        =&\deg_\mathcal{Y}(h) + \deg_\mathcal{Y}(\mathcal{X}\mathcal{Y}^{q} - \mathcal{X}^q\mathcal{Y} - 1)&\\
        >&\deg_\mathcal{Y}(h).&
    \end{align*}
    Thus, we have arrived at a contradiction, and are forced to conclude ${F=0}$.\\
    \\
    We now count how many polydifferentials are in ${\mathcal{S}}$. We have
    \begin{align*}
        |\mathcal{S}|&=\sum_{j=0}^{q-1}\left( m(q-2)-j+1\right) + m(q-2) - q + 1&\\
        & = (2m-1)(q(q-1)/2-1)&\\\
        & = (2m-1)(g(C)-1)&\text{(by Proposition \ref{genus of C})}\\
        & = \dim_{\mathbb{F}}H^0(C,\Omega_C^{\otimes m})&\text{(by (\ref{dimension of globally holomorphic polydifferentials}))}.
    \end{align*}
    Since we have found as many linearly independent polydifferentials as the dimension of ${H^0(C,\Omega_C^{\otimes m})}$, we have found a basis.
\end{proof}
Now that we have a basis for ${H^0(C,\Omega_C^{\otimes m})}$, we will describe the action of ${G = SL_2(\mathbb{F}_q)}$ on these basis elements.
\begin{lem}
    \label{action of element on holomorphic polydifferentials}
    Let ${i,j \in \mathbb{Z}}$ and let ${\omega_{ij}}$ be as in (\ref{omega_ij definition}). Furthermore, let
    ${
        \sigma=
        \left(
            \begin{array}{ll}
                 \alpha&\beta  \\
                 \gamma&\delta 
            \end{array}
        \right) \in G
    }$. Then we have:
    $$
    {
        \omega_{ij}\cdot
        \sigma
        =
        \frac{(\alpha x + \beta y)^i(\gamma x + \delta y)^j}{x^{mq}}dx^{\otimes m}.
    }
    $$
\end{lem}
\begin{proof}
    This proof will be similar to that of \cite[\S 4, Lem. 4.1]{lucas}. Note that
    $$
    {
        dx^{\otimes m}\cdot \sigma = d(\alpha x + \beta y)^{\otimes m} = \left(\alpha dx + \beta dy\right)^{\otimes m}.
    }
    $$
    Since ${dy = (y/x)^q dx}$ (see proof of Proposition \ref{basis for holomorphic polydifferentials}), we have
    \begin{align*}
        \left(\alpha dx + \beta dy\right)^{\otimes m} &= \left(\alpha dx + \beta \left(\frac{y}{x}\right)^qdx\right)^{\otimes m} = \left(\frac{\alpha x^q + \beta y^q}{x^q}dx\right)^{\otimes m} = \frac{(\alpha x^q + \beta y^q)^m}{x^{mq}}dx^{\otimes m} \\
        &= \frac{(\alpha x + \beta y)^{mq}}{x^{mq}}dx^{\otimes m}
    \end{align*}
    where the final equality follows from the fact that $q$ is a power of ${p = \Char(\mathbb{F})}$. Thus, altogether we obtain
    $$
    {
        \omega_{ij}\cdot \sigma = \frac{x^iy^j}{x^{mq}}dx^{\otimes m}\cdot \sigma = \frac{(\alpha x + \beta y)^i(\gamma x + \delta y)^j}{(\alpha x + \beta y)^{mq}}\frac{(\alpha x + \beta y)^{mq}}{x^{mq}}dx^{\otimes m} = \frac{(\alpha x + \beta y)^i(\gamma x + \delta y)^j}{x^{mq}}dx^{\otimes m}
    }
    $$
    as claimed.
\end{proof}
We will now define a set of subspaces ${(W_k)_{k=0}^{q}}$ of ${H^0(C,\Omega_C^{\otimes m})}$ and show that ${H^0(C,\Omega_C^{\otimes m})\cong \bigoplus_{k=0}^{q}W_k}$. In order to do this, we will need a lemma and a definition.
\begin{cor}[Reduction formula]
    \label{reduction lemma}
    Suppose ${q\leq j\leq m(q-2)}$ and ${1\leq i\leq m(q-2)-j}$. Then we have
    $$
    {
        \omega_{ij} = \omega_{i-1+q,j-q+1} + \omega_{i-1,j-q}
    }
    $$
    where both ${\omega_{i-1+q,j-q+1}}$ and ${\omega_{i-1,j-q}}$ are holomorphic polydifferentials.
\end{cor}
\begin{proof}
    Equality of both sides can be derived by multiplying both sides of the equation ${xy^q = x^qy + 1}$ by ${x^{i-1}y^{j-q}}$. The only thing left to check is that the two right summands are again holomorphic. This follows from Proposition \ref{basis for holomorphic polydifferentials}, because the assumptions on ${i,j}$ imply the inequalities
    $$
    {
        \begin{array}{lllll}
             0&\leq&j-q+1&\leq&m(q-2)\\
             0&\leq&i-1+q&\leq&m(q-2)-j+q-1\\
             0&\leq&j-q&\leq&m(q-2)\\
             0&\leq&i-1&\leq&m(q-2)-j+q
        \end{array}.
    }
    $$
\end{proof}
\begin{defn}
    \label{degree of a polydifferentials}
    For ${0\leq i,j\leq m(q-2)}$, ${0\leq i+j\leq m(q-2)}$, we define the \textit{degree} of the holomorphic polydifferential ${\omega_{ij}}$ as ${\deg(\omega_{ij}) = i+j\mod{q+1} \in \mathbb{Z}/(q+1)\mathbb{Z}}$.
\end{defn}
If ${\omega_{ij}}$ is a holomorphic polydifferential that is not an element of our basis ${\mathcal{S}}$, by recursive application of Corollary \ref{reduction lemma}, ${\omega_{ij}}$ can be written as a linear combination of elements of ${\mathcal{S}}$ of the same degree as ${\omega_{ij}}$.
\begin{prop}
    \label{partial decomposition of polydifferentials}
    For ${0\leq k\leq q}$, let
    \begin{align*}
        \mathcal{S}_k &:= \{\omega_{ij} \in \mathcal{S}: \deg(\omega_{ij}) \equiv k\Mod{q+1}\}\\
        W_k &:= \Span_{\mathbb{F}}\mathcal{S}_k
    \end{align*}
    where ${\mathcal{S}}$ is given as in (\ref{basis set for holomorphic polydifferentials}). Then each ${W_k}$ is an ${\mathbb{F}[G]}$-submodule of ${H^0(C,\Omega_C^{\otimes m})}$, and furthermore
    $$
    {
        H^0(C,\Omega_C^{\otimes m}) \cong \bigoplus_{k=0}^{q}W_k.
    }
    $$
\end{prop}
\begin{proof}
    Note that since each ${\omega_{ij} \in \mathcal{S}}$ occurs in exactly one ${\mathcal{S}_k}$, we have:
    \begin{enumerate}
        \item For ${0\leq \kappa < l\leq q}$, ${\mathcal{S}_{\kappa}\cap \mathcal{S}_{l} = \emptyset}$.
        \item ${\bigcup_{k=0}^{q}\mathcal{S}_k = \mathcal{S}}$.
    \end{enumerate}
    Thus we have the decomposition given in the Proposition statement when considering both sides as vector spaces over ${\mathbb{F}}$. It remains to argue that each ${W_k}$ is $G$-invariant.\\
    For ${\sigma \in G}$ and ${\omega_{ij} \in \mathcal{S}}$, by Lemma \ref{action of element on holomorphic polydifferentials} we see that ${\omega_{ij}\cdot \sigma}$ can be written as a sum of holomorphic polydifferentials of the same degree as ${\omega_{ij}}$. By recursive application of Corollary \ref{reduction lemma}, these polydifferentials can then be written as a linear combination of elements from our basis ${\mathcal{S}}$ with the same degree. We thus conclude that each ${W_k}$ is a ${G}$-invariant subspace of ${H^0(C,\Omega_C^{\otimes m})}$.
\end{proof}

\section{The modular representation theory of $G$ and $B$}\label{preliminary-section}
In this section, we will recall some results regarding the modular representation theory of ${G = SL_2(\mathbb{F}_p)}$ and its subgroup of upper triangular matrices $B$. We make use of the notations given in \S \ref{notation-section}. From this section onwards, we assume that ${q=p}$.
\begin{prop}[Parameterisation of the indecomposable ${\mathbb{F}[B]}$ modules]
    \label{parameterisation of F[B] modules}
    For ${a \in \mathbb{Z}}$, let ${S_a}$ denote the one dimensional vector space over ${\mathbb{F}}$ on which the generator ${\left(\begin{matrix}\zeta&0\\ 0&\zeta^{-1}\end{matrix}\right)}$ of $T$ acts as multiplication by ${\zeta^a}$. Then ${S_0, S_1, ..., S_{p-2}}$ are precisely the simple modules for ${\mathbb{F}[T]}$. We regard them also as ${\mathbb{F}[B]}$ modules via inflation along the quotient map ${B \twoheadrightarrow U}$, i.e. by letting $U$ act trivially. Next, let ${1\leq b \leq p}$. Then there exists a uniserial ${\mathbb{F}[B]}$ module ${U_{a,b}}$ of dimension $b$ whose composition factors are given by ${S_a, S_{a+2}, S_{a+4}, ...}$ in ascending order. This describes all indecomposable ${\mathbb{F}[B]}$ modules. The projective indecomposables are given by ${U_{a,p}}$.
\end{prop}
\begin{proof}
    The projective indecomposable ${\mathbb{F}[B]}$ modules and their composition factors follow from \cite[\S \RN{2}.5, pp. 35-37]{localrep}, \cite[\S \RN{2}.5, pp. 37, ex. 3]{localrep}. It follows from \cite[\S \RN{2}.6, pp. 42-43]{localrep} that all indecomposable ${\mathbb{F}[B]}$ modules are uniserial and obtained as homomorphic images of projective indecomposables, giving the full description.
\end{proof}
\begin{prop}[The simple ${\mathbb{F}[G]}$ modules]
    \label{simple F[G] modules}
    Let ${1\leq t\leq p}$. Define the vector subspace ${V_t\leq \mathbb{F}[\mathcal{X},\mathcal{Y}]}$ as
    $$
    {
        V_t = \Span_{\mathbb{F}}\{\mathcal{X}^{t-1},\ \mathcal{X}^{t-2}\mathcal{Y},\ ...,\ \mathcal{X}\mathcal{Y}^{t-2},\ \mathcal{Y}^{t-1}\}.
    }
    $$
    Then under the action
    $$
    {
        \mathcal{X}^i\mathcal{Y}^j \cdot
        \left(
            \begin{array}{ll}
                \alpha&\beta\\
                \gamma&\delta
            \end{array}
        \right)
        :=
        (\alpha \mathcal{X} + \beta \mathcal{Y})^i(\gamma \mathcal{X} + \delta \mathcal{Y})^j,
    }
    $$
    ${V_1,\ V_2,\ ...,\ V_p}$ gives a complete list of the simple ${\mathbb{F}[G]}$ modules up to isomorphism.
\end{prop}
\begin{proof}
    See \cite[\S 10.1.2]{bonnafe} or \cite[pp. 14-16]{localrep}.
\end{proof}
As a consequence of basic results from modular representation theory, we have a parameterisation of the projective indecomposable modules as ${P_{V_t}}$ for ${1\leq t\leq p}$. We now describe the structure of these ${\mathbb{F}[G]}$ modules.
\begin{prop}
    \label{projective indecomposable F[G] modules}
    Let ${1\leq t\leq p}$, and let ${P_{V_t}}$ denote the projective cover of the simple module ${V_t}$.
    \begin{itemize}
        \item ${\underline{t=1}}$. ${P_{V_1}}$ is a uniserial module of dimension $p$, with composition factors ${V_1, V_{p-2}, V_1}$ in ascending order.
        \item ${\underline{1 < t < p-1}}$. ${P_{V_t}}$ is of dimension ${2p}$ and has the three socle layers ${V_t, V_{p+1-t} \oplus V_{p-1-t}, V_t}$.
        \item ${\underline{t=p-1}}$. ${P_{V_{p-1}}}$ is of dimension ${2p}$ and has the three socle layers ${V_{p-1}, V_{2}, V_{p-1}}$.
        \item ${\underline{t=p}}$. ${P_{V_p} \cong V_p}$ is both simple and projective.
    \end{itemize}
\end{prop}
\begin{proof}
    See \cite[\S \RN{2}.7]{localrep}.
\end{proof}
We now move onto giving a parameterisation of the non-projective indecomposable ${\mathbb{F}[G]}$ modules. For this, we use the fact that the group algebras ${\mathbb{F}[G]}$, ${\mathbb{F}[B]}$ are in \textit{Green correspondence}, which we now give a statement of.
\begin{defn}
    \label{trivial intersection subgroup}
    A Sylow-$p$ subgroup ${\mathcal{P}}$ of a finite group ${\mathcal{G}}$ is called a \textit{trivial intersection subgroup} if for every ${x \in \mathcal{G}}$, we have ${\mathcal{P}\cap x^{-1}\mathcal{P}x \in \{1,\mathcal{P}\}.}$
\end{defn}
\begin{thm}
    \label{green correspondence}
    (Special case of Green correspondence). Let $k$ denote an algebraically closed field of characteristic $p$. Let $\mathcal{P}$ be a trivial intersection Sylow-$p$ subgroup of ${\mathcal{G}}$, and let ${\mathcal{H} = N_{\mathcal{G}}(\mathcal{P})}$. Then there is a bijective correspondence between the non-projective indecomposable ${k[\mathcal{G}]}$ and ${k[\mathcal{H}]}$ modules given by restriction and induction. More precisely, given a non-projective indecomposable ${k[\mathcal{G}]}$ module $M$, ${\Res^{\mathcal{G}}_{\mathcal{H}}(M) \cong N \oplus Q}$ where ${N}$ is a non-projective indecomposable ${k[\mathcal{H}]}$ module and ${Q}$ is a projective ${k[\mathcal{H}]}$ module. Conversely, ${\ind_{\mathcal{H}}^{\mathcal{G}}(N) \cong M \oplus P}$, where $P$ is a projective ${k[\mathcal{G}]}$ module.
\end{thm}
\begin{proof}
    See \cite[\S \RN 3.10, Thm. 1]{localrep}.
\end{proof}
In particular, note that ${B = N_G(U)}$ where $U$ is a Sylow-$p$ subgroup of $G$ of order $p$ and is thus a trivial intersection subgroup. We have by Proposition \ref{parameterisation of F[B] modules} that the non-projective indecomposable ${\mathbb{F}[B]}$ modules are given by ${U_{a,b}}$ for ${0\leq a\leq p-2}$, ${1\leq b\leq p-1}$. Thus, defining ${V_{a,b}}$ to be the Green correspondent of ${U_{a,b}}$ (i.e. the unique non-projective indecomposable summand of ${\Ind^G_B(U_{a,b})}$) parameterizes all the non-projective indecomposable ${\mathbb{F}[G]}$ modules.\\
\\
The following result takes an ${\mathbb{F}[G]}$ module $M$ such that:
\begin{enumerate}
    \item the decomposition of ${\Res^G_B(M)}$ as a direct sum of indecomposable ${\mathbb{F}[B]}$ modules is known
    \item the composition factors of ${M}$ as an ${\mathbb{F}[G]}$ module are known
\end{enumerate}
and computes the decomposition of $M$ as a direct sum of indecomposable ${\mathbb{F}[G]}$ modules.
\begin{thm}
    \label{lifting decomposition}
    Let $M$ be an ${\mathbb{F}[G]}$ module. For ${1\leq t\leq p}$ let ${\ell_t}$ denote the multiplicity of ${V_t}$ as a composition factor of ${M}$. Let ${\Res^G_B(M) \cong \bigoplus_{a=0}^{p-2}\bigoplus_{b=1}^{p}U_{a,b}^{\oplus n_{a,b}}}$ for some multiplicities ${n_{a,b} \in \mathbb{N}}$. For ${a \in [0,p-2]}$, ${b \in [1,p-1]}$ and ${t \in [1,p]}$, let ${c_{a,b,t}}$ denote the multiplicity of ${V_t}$ as a composition factor of ${V_{a,b}}$. Define
    \begin{equation}
    \label{alpha_t}
        \alpha_t =
        \begin{dcases}
            \ell_t - \sum_{a=0}^{p-2}\sum_{b=1}^{p-1}n_{a,b}c_{a,b,t}&\text{if }t \in [1,p-1]\\
            \ell_p&\text{if }t=p
        \end{dcases}
    \end{equation}
    and for ${1\leq i,j\leq (p-1)/2}$, define
    \begin{equation}
    \label{gamma_i,j}
        \Gamma_{i,j} =
        \left\{
            \begin{array}{ll}
                 (-1)^{i+j}\left(i - \frac{2ij}{p}\right)&\text{if }i\leq j  \\
                 (-1)^{i+j}\left(j - \frac{2ij}{p}\right)&\text{if }i > j 
            \end{array}
        \right. .
    \end{equation}
    Finally, let ${n_p = \alpha_p}$, and let ${n_t}$ for ${t \in \left[1,p-1\right]}$ be given by
    \begin{equation}
        \label{n_t definition}
        n_t
        =
        \begin{dcases}
            \sum_{1\leq j\leq (p-1)/2\text{ odd}}\Gamma_{t,j}\alpha_j +
            \sum_{1\leq j\leq (p-1)/2\text{ even}}\Gamma_{t,j}\alpha_{p-j}&\text{if }t \in \left[1,\frac{p-1}{2}\right]\text{ odd}\\
            \sum_{1\leq j\leq (p-1)/2\text{ odd}}\Gamma_{p-t,j}\alpha_j +
            \sum_{1\leq j\leq (p-1)/2\text{ even}}\Gamma_{p-t,j}\alpha_{p-j}&\text{if }t \in \left[\frac{p+1}{2},p-1\right]\text{ odd}\\
            \sum_{1\leq j\leq (p-1)/2\text{ odd}}\Gamma_{t,j}\alpha_{p-j} +
            \sum_{1\leq j\leq (p-1)/2\text{ even}}\Gamma_{t,j}\alpha_{j}&\text{if }t \in \left[1,\frac{p-1}{2}\right]\text{ even}\\
            \sum_{1\leq j\leq (p-1)/2\text{ odd}}\Gamma_{p-t,j}\alpha_{p-j} +
            \sum_{1\leq j\leq (p-1)/2\text{ even}}\Gamma_{p-t,j}\alpha_{j}&\text{if }t \in \left[\frac{p+1}{2},p-1\right]\text{ even}\\
        \end{dcases}.
    \end{equation}
    Then
    $$
    {
        M \cong \bigoplus_{a=0}^{p-2}\bigoplus_{b=1}^{p-1}V_{a,b}^{\oplus n_{a,b}} \oplus \bigoplus_{t=1}^{p}P_{V_i}^{\oplus n_t}.
    }
    $$
\end{thm}
\begin{proof}
    This is \cite[\S 4, Thm. 4.3]{marchment2025green}.
\end{proof}
\begin{rem}
    \label{F[G] module decomposition of H^0}
    When applied to ${M = H^0(C,\Omega_C^{\otimes m})}$ this theorem, in principle, solves our problem of computing the multiplicity of each indecomposable ${\mathbb{F}[G]}$ module in $M$. However, we still need to compute the quantities ${c_{a,b,t}}$, ${n_{a,b}}$ and ${\ell_t}$. For ${c_{a,b,t}}$, this has been done in \cite[\S 3, Cor. 3.18]{marchment2025green} and the next theorem recalls the resulting formulae. For ${n_{a,b}}$ and ${\ell_t}$, this is achieved in Theorem \ref{decomposition as FB module, q=p} and Theorem \ref{composition factors of H^0}, respectively.
\end{rem}
\begin{thm}
    \label{c_{a,b,t} coefficients}
    Let ${b \in [1,p-1]}$. If ${a \in [0,1]}$, define
    $$
    {
        \begin{array}{ll}
             c_{a,b}(t)=&\ \ 
             \left\{
             \begin{array}{lll}
                  1&\text{if }b \in \left[1,\frac{p-1}{2}\right],&t \in \left[1,2b+a-1\right]  \\
                  1&\text{if }b \in \left[\frac{p+1}{2},p-1\right],&t \in \left[1,2(p-b)-a\right]\\
                  0&\text{otherwise}
             \end{array}
             \right.
             \\
             &\\
             &+\left\{
             \begin{array}{lll}
                  1&\text{if }b \in \left[1,\frac{p-1}{2}\right],&t \in \left[p-a-2b+1,p-1\right]  \\
                  1&\text{if }b \in \left[\frac{p+1}{2},p-1\right],&t \in \left[a+2b-p,p-1\right]\\
                  0&\text{otherwise}
             \end{array}
             \right. .
        \end{array} 
    }
    $$
    Next, if ${a \in \left[2,\frac{p-1}{2}\right]}$, then define
    $$
    {
        \begin{array}{ll}
             c_{a,b}(t)=&\ \ 
             \left\{
             \begin{array}{lll}
                  1&\text{if }b \in \left[1,\frac{p-a}{2}\right],&t \in \left[a,a-1+2b\right] \\
                  1&\text{if }b \in \left[\frac{p-a+1}{2},p-a\right],&t \in \left[a,2(p-b)-a\right] \\
                  1&\text{if }b \in \left[p-a,p-\frac{a+1}{2}\right],&t \in \left[2(p-b)-a,a\right]  \\
                  1&\text{if }b \in \left[p-\frac{a}{2},p-1\right],&t \in \left[2(b-p)+a+1,a\right] \\
                  0&\text{otherwise}
             \end{array}
             \right.
             \\
             &\\
             &+\left\{
             \begin{array}{lll}
                  1&\text{if }b \in \left[1,\frac{p-a}{2}\right],&t \in \left[p-a-2b+1,p-1\right]  \\
                  1&\text{if }b \in \left[\frac{p-a+1}{2},p-a\right],&t \in \left[a+2b-p,p-1\right] \\
                  1&\text{if }b \in \left[p-a,p-\frac{a+1}{2}\right],&t \in \left[p-a,p-1\right]  \\
                  1&\text{if }b \in \left[p-\frac{a}{2},p-1\right],&t \in \left[p-a,p-1\right] \\
                  0&\text{otherwise}
             \end{array}
             \right. .
        \end{array} 
    }
    $$
    Finally, if ${a \in \left[\frac{p+1}{2},p-2\right]}$, define
    $$
    {
        \begin{array}{ll}
             c_{a,b}(t)=&\ \ 
             \left\{
             \begin{array}{lll}
                  1&\text{if }b \in \left[1,\frac{p-a}{2}\right],&t \in \left[p-a-2b+1,p-a\right]  \\
                  1&\text{if }b \in \left[\frac{p-a+1}{2},p-a\right],&t \in \left[a+2b-p,p-a\right] \\
                  1&\text{if }b \in \left[p-a,p-\frac{a+1}{2}\right],&t \in \left[p-a,a+2b-p\right]  \\
                  1&\text{if }b \in \left[p-\frac{a}{2},p-1\right],&t \in \left[p-a,2(p-b)-a-1+p\right] \\
                  0&\text{otherwise}
             \end{array}
             \right.
             \\
             &\\
             &+\left\{
             \begin{array}{lll}
                  1&\text{if }b \in \left[1,\frac{p-a}{2}\right],&t \in \left[a,p-1\right]  \\
                  1&\text{if }b \in \left[\frac{p-a+1}{2},p-a\right],&t \in \left[a,p-1\right] \\
                  1&\text{if }b \in \left[p-a,p-\frac{a+1}{2}\right],&t \in \left[2(p-b)-a,p-1\right]  \\
                  1&\text{if }b \in \left[p-\frac{a}{2},p-1\right],&t \in \left[2(b-p)+a+1,p-1\right] \\
                  0&\text{otherwise}
             \end{array}
             \right. .
        \end{array} 
    }
    $$
    Then the number of times the simple module ${V_t}$ occurs as a composition factor of the Green correspondent ${V_{a,b}}$ of ${U_{a,b}}$ is given by:
    $$
    {
        c_{a,b,t}
        =
        \left\{
            \begin{array}{ll}
                 0&\text{if }t \equiv a\Modwb{2}\\
                 c_{a,b}(t)&\text{if }t\not\equiv a\Modwb{2},\ t \in \left[1,\frac{p-1}{2}\right]  \\
                 c_{a,b}(p-t)&\text{if }t\not\equiv a\Modwb{2},\ t \in \left[\frac{p+1}{2},p-1\right] 
            \end{array}
        \right. .
    }
    $$
\end{thm}
We will also make use of the following result, which describes the composition factors of the induced modules ${\Ind^G_B(S_a)}$.
\begin{prop}
    \label{Ind_B^G(S_a) composition factors}
    Let ${0\leq a\leq p-2}$. The composition factors of ${\Ind_B^G(S_a)}$ are ${V_{a+1},V_{p-a}}$. 
\end{prop}
\begin{proof}
    This is \cite[\S 3, Cor. 3.17]{marchment2025green}.
\end{proof}

\section{${\mathbb{F}[B]}$ module decomposition of ${H^0(C,\Omega_C^{\otimes m})}$}
In this section, we give the decomposition of ${\Res^G_B(H^0(C,\Omega_C^{\otimes m}))}$ when ${q=p}$. As before, we make use of the notations in \S \ref{notation-section}. Since ${B = U\rtimes T}$ is a \textit{p-hypo-elementary group}, we can utilise Theorem 1.6 from \cite{bleher} to write ${H^0(C,\Omega_C^{\otimes m})}$ as a direct sum of indecomposable ${\mathbb{F}[B]}$ modules. In order to apply this theorem, we first recall and compute various items.\\
\\
The following Proposition holds in general for any prime power $q$, and so do Lemma \ref{quotient of p1(F) by T} and Corollary \ref{the quotient C/B} after the obvious modifications.
\begin{prop}
    \label{the quotient C/U}
    We have ${C/U \cong \mathbb{P}^1(\mathbb{F})}$, with quotient map
    $$
    {
        \pi_U(X:Y:Z)
        =
        \left\{
        \begin{array}{ll}
             [Y:Z]&\text{when ${Y\neq 0}$}  \\
            
             [0:1]&\text{when ${Y=0}$} 
        \end{array}
        \right. .
    }
    $$
    Furthermore, the only point of ${C}$ ramified under the action of $U$ is ${[1:0:0]}$, which has ramification index ${q}$. Finally, the induced action of ${B/U \cong T}$ on ${C/U \cong \mathbb{P}^1(\mathbb{F})}$ is given by
    $$
    {
        \left(
            \begin{array}{ll}
                 \alpha&0  \\
                 0&\alpha^{-1} 
            \end{array}
        \right)
        \cdot
        [Y:Z]
        =
        [Y:\alpha Z].
    }
    $$
\end{prop}
\begin{proof}
    The first two statements are \cite[\S 1.2, Lem. 1.30]{lucas-mphil}. We show that the induced action of ${B/U \cong T}$ on ${C/U \cong \mathbb{P}^1(\mathbb{F})}$ is as above. Let ${[Y:Z] \in \mathbb{P}^1(\mathbb{F})}$ with ${Y \neq 0}$. Because ${\mathbb{F}}$ is algebraically closed, there is some ${X' \in \mathbb{F}}$ so that ${X'Y^{p} - (X')^pY - Z^{p+1} = 0}$. Hence we can lift ${[Y:Z]}$ to a point ${[X':Y:Z]}$ under ${\pi_U}$ and:
    $$
    {
        \left(
            \begin{array}{ll}
                 \alpha&0  \\
                 0&\alpha^{-1} 
            \end{array}
        \right)
        \cdot
        [X:Y]
        =
        \pi_U\left(
            \left(
            \begin{array}{ll}
                 \alpha&0  \\
                 0&\alpha^{-1} 
            \end{array}
            \right)
            \cdot
            [X':Y:Z]
        \right)
    }
    $$
    which is
    $$
    {
        \pi_U([\alpha X':\alpha^{-1}Y:Z])
        =
        [Y:\alpha Z].
    }
    $$
    The case of ${[Y:Z] = [0:1]}$ is done in the same way.
\end{proof}
\begin{lem}
    \label{quotient of p1(F) by T}
    We have ${\mathbb{P}^1(\mathbb{F})/T \cong \mathbb{P}^1(\mathbb{F})}$, with quotient morphism ${\lambda_T: \mathbb{P}^1(\mathbb{F}) \to \mathbb{P}^1(\mathbb{F})}$ given by:
    $$
    {
        \lambda_T(Y:Z) = [Y^{p-1}:Z^{p-1}].
    }
    $$
    Furthermore, the only points in ${\mathbb{P}^1(\mathbb{F})}$ ramified under the action of $T$ are ${[0:1],[1:0]}$. The ramification indices are given by ${e_{[1:0]} = e_{[0:1]} = p-1}$. Finally, let ${\theta}$ denote the character of the simple ${\mathbb{F}[T]}$ module ${S_1}$ (see Proposition \ref{parameterisation of F[B] modules}). Then the fundamental characters ${\theta_{[1:0]},\theta_{[0:1]}}$ of ${\mathbb{P}^1(\mathbb{F})}$ under the action of $T$ are given by ${\theta,\theta^{p-2}}$ respectively.
\end{lem}
\begin{proof}
    These results are standard and easy, and thus we omit the proofs.
\end{proof}
\begin{cor}
    \label{the quotient C/B}
    We have ${C/B \cong \mathbb{P}^1(\mathbb{F})}$, with quotient map
    $$
    {
        \pi_B([X:Y:Z])
        =
        \begin{dcases}
            [Y^{p-1}:Z^{p-1}]&\text{when }Y\neq 0\\
            [0:1]&\text{when }Y=0.
        \end{dcases}
    }
    $$
    The point of ramification under the action of $B$ are of the form ${[X:Y:0]}$. The point ${[1:0:0]}$ has ramification index ${p(p-1))}$ under $B$, and all other points at infinity have ramification index ${p-1}$ under $B$.
\end{cor}
\begin{proof}
    This now follows from Proposition \ref{the quotient C/U} and Lemma \ref{quotient of p1(F) by T}.
\end{proof}
The next lemma evaluates the quantity ${d_{P,j}}$ which, in a more general context, is defined in \cite{bleher} for any ${0\leq j\leq p-1}$ and ${P \in C/U}$ as follows:
\begin{equation}
    \label{d_{P,j} definition}
    d_{P,j} = \floor*{\frac{m\sum_{i\geq 0}(|I_{Q,i}|-1) - \sum_{l=1}^{n_Q}a_{l,t}p^{n_Q-l}j_l(Q)}{p^{n_Q}}}
\end{equation}
where (1) $Q$ is any point in ${\pi^{-1}_U(P)}$, (2) $I$ is defined as the cyclic subgroup of $U$ generated by the Sylow $p$-subgroups of the inertia groups of all closed points of $C$, (3) ${n_Q\geq 0}$ is defined as the natural number such that ${|I_Q|=p^{n_Q}}$, (4) ${a_{l,t}}$ is the ${l^{\thh}}$ coefficient of the $p$-adic expansion of the unique natural number ${t}$ satisfying ${p^{n_I-n_Q}t \leq j\leq p^{n_I-n_Q}(t+1)}$, where ${|I| = p^{n_I}}$, (5) ${j_l(Q)}$ are the jumping numbers of $Q$ under the action of ${I_Q}$.\\
\\
For ${0\leq j\leq p-1}$, let
\begin{equation}
    \label{D_j definition}
    D_j = \sum_{P \in C/U}d_{P,j}P
\end{equation}
with ${d_{P,j}}$ given as in (\ref{d_{P,j} definition}).
\begin{lem}
    \label{D_j lemma}
     Let ${0\leq j\leq p-1}$. After identifying ${C/U}$ with ${\mathbb{P}^1(\mathbb{F})}$ according to Proposition \ref{the quotient C/U}, we obtain:
    $$
    {
        D_j = \left(m(p+1) - j - \ceil*{\frac{2m+j}{p}}\right)[0:1].
    }
    $$
\end{lem}
\begin{proof}
    We have ${I = U}$. Hence ${n_I = 1}$. Next, ${n_Q = 1}$ precisely when ${Q}$ is a ramified point under the action of $U$, and ${n_Q = 0}$ otherwise. It's clear that when ${n_Q = 0}$, we have ${d_{P,j} = 0}$. When ${n_Q = 1}$, clearly ${t = j}$, and hence ${a_{1,t} = j}$ and ${a_{l,t} = 0}$ otherwise. By Proposition \ref{the quotient C/U}, the only point ramified under the action of $U$ is ${Q = [1:0:0]}$. By \cite[\S 2.1, Lem. 2.6]{lucas-mphil}, we have ${j_1([1:0:0]) = p+1}$. Putting this altogether, we obtain
    $$
    {
        D_j = \floor*{\frac{m(p-1)(p+2) - j(p+1)}{p}}\pi_U([1:0:0]) =
        \floor*{
            m(p+1) - j - \frac{2m+j}{p}
        }[0:1].
    }
    $$
    By using the fact that ${\floor{-s} = -\ceil{s}}$ for real numbers $s$, we arrive at what is stated in the lemma.
\end{proof}
Next, we fix the canonical divisor ${\mathcal{C} := -2[0:1]}$ on ${C/U \cong \mathbb{P}^1(\mathbb{F})}$. For ${0\leq j\leq p-1}$, as in \cite{bleher} we define the following:
\begin{equation}
    \label{K definition}
    \mathcal{K} = \lambda_T^*(\mathcal{C})
    + \Ram_{\lambda_T}
\end{equation}
\begin{equation}
    \label{E_j definition}
    E_j = m\mathcal{K}  + D_j.
\end{equation}
Note that by \cite[\S \RN{4}.2, Prop. 2.3]{hartshorne}, ${\mathcal{K}}$ is a canonical divisor on ${\mathbb{P}^1(\mathbb{F})}$.
\begin{lem}
    \label{E_j lemma}
    For ${0\leq j\leq p-1}$, ${E_j}$ as defined in equation (\ref{E_j definition}) is given by:
    $$
    {   
        E_j = m(p-2)[1:0] + \left(m-j-\ceil*{\frac{2m+j}{p}}\right) [0:1].
    }
    $$
\end{lem}
\begin{proof}
    By Corollary \ref{the quotient C/U}, we get ${\lambda_T^*(\mathcal{C}) = -2(p-1)[0:1]}$, and also that all points in ${C/U \cong \mathbb{P}^1(\mathbb{F})}$ are tamely ramified under ${\lambda_T}$. By \cite[\S \RN{4}.2, Prop. 2.2]{hartshorne},
    $$
    {
        \Ram_{\lambda_T} = (p-2)[1:0] + (p-2)[0:1].
    }
    $$
    Adding these together, we get
    $$
    {
        \mathcal{K} = (p-2)[1:0] - p[0:1].
    }
    $$
    Therefore, ${E_j}$ is
    $$
    {
        E_j = m(p-2)[1:0] - mp[0:1] + \left(m(p+1) - j - \ceil*{\frac{2m+j}{p}}\right)[0:1]
    }
    $$
    which indeed simplifies to the lemma statement.
\end{proof}
As also defined in \cite{bleher}, for ${0\leq j\leq p-1}$, we let ${0\leq \ell_{[1:0],j},\ell_{[0:1],j}\leq p-2}$ be such that
\begin{equation}
    \label{definition of l_j}
    \begin{aligned}
             \ell_{[1:0],j}\equiv&\ m(p-2)\Mod{p-1} \\
             \ell_{[0:1],j}\equiv&\ \left(m-j-\ceil*{\frac{2m+j}{p}}\right)\Mod{p-1}.
    \end{aligned}
\end{equation}
Although ${\ell_{[1:0],j}}$ as above does not depend on the index $j$, we keep it to match the notation of Theorem ${1.6}$ in \cite{bleher}. Next, as in \cite{bleher}, for each branch point ${z \in \Br(\lambda_T) \subset C/B \cong \mathbb{P}^1(\mathbb{F})}$, let ${y(z) \in C/U \cong \mathbb{P}^1(\mathbb{F})}$ such that ${\lambda_T(y(z)) = z}$. Let
\begin{equation}\label{n_j definition}
    n_j = 1 - g\left(C/B\right)
    +
    \sum_{z \in \Br(\lambda_T)}
    \frac{\Ord_{y(z)}(E_j) - \ell_{y(z),j}}{e_{y(z)}}
\end{equation}
where ${e_{y(z)}}$ is the ramification index for ${y(z)}$ under the action of ${T}$.
\begin{lem}
    \label{n_j lemma}
    We have that ${n_j}$ as defined in (\ref{n_j definition}) is given by
    $$
    {
        n_j = 
        1 + m -
        \ceil*{\frac{m}{p-1}}
        +
        \floor*{\frac{m-j-\ceil*{\frac{2m+j}{p}}}{p-1}} .
    }
    $$
\end{lem}
\begin{proof}
    Since the genus of the projective line is $0$, substituting Lemma \ref{E_j lemma} for the two branch points we have under ${\lambda_T}$ yields
    $$
    {
        n_j = 1 + \floor*{\frac{m(p-2)}{p-1}} + \floor*{\frac{m - j  - \ceil*{\frac{2m+1}{j}}}{p-1}}
    }
    $$
    and hence the result.
\end{proof}
We continue following \cite{bleher}. Let ${P \in C/U \cong \mathbb{P}^1(\mathbb{F})}$, and let ${\theta_P}$ be the fundamental character at $P$ under the action of $T$. Then for ${0\leq a\leq p-1}$ and ${i \in \mathbb{Z}}$, we define
\begin{equation}
    \label{mu_{a,i} definition}
    \mu_{a,i}(P) =
    \left\{
    \begin{array}{ll}
         1&\text{if }\theta_P^i \text{ is equal to the character of }\Res^T_{T_P}(S_a)  \\
         0&\text{otherwise} 
    \end{array}
    \right. .
\end{equation}
(where ${S_a}$ is defined in Proposition \ref{parameterisation of F[B] modules}). Note that if ${P}$ is an unramified point, then since ${T_P = 1}$, we have ${\mu_{a,i}(P) = 1}$. The non-trivial behaviour of ${\mu_{a,i}}$ occurs for the ramified points of ${C/U \cong \mathbb{P}^1(\mathbb{F})}$.
\begin{lem}
    \label{mu_{a,i} lemma}
    For the two ramified points ${P \in \{[1:0],[0:1]\}}$ of ${C/U \cong \mathbb{P}^1(\mathbb{F})}$ under the action of $T$ (Corollary \ref{the quotient C/U}), ${\mu_{a,i}(P)}$ as defined in (\ref{mu_{a,i} definition}) gives us the following two functions:
    $$
    {
        \mu_{a,i}([1:0]) =
        \left\{
            \begin{array}{ll}
                 1&\text{if }a\equiv i\Mod{p-1}  \\
                 0&\text{otherwise} 
            \end{array}
        \right.
        ,\ 
        \mu_{a,i}([0:1]) =
        \left\{
            \begin{array}{ll}
                 1&\text{if }a\equiv (p-1-i)\Mod{p-1}  \\
                 0&\text{otherwise} 
            \end{array}
        \right. .
    }
    $$
\end{lem}
\begin{proof}
    We have ${\theta_{[1:0]} = \theta}$, so ${\theta_{[1:0]}^i = \theta^i}$. Also, ${T_{[1:0]} = T}$, and so ${\Res^T_{T_{[1:0]}}(S_a) = S_a}$, which has character ${\theta^a}$. Then ${\theta^i = \theta^a}$ if and only if ${a\equiv i\Mod{p-1}}$. Similarly, ${\theta_{[0:1]}^i = \theta^{(p-2)i}}$, and ${\theta^a = \theta^{(p-2)i}}$ if and only if ${a \equiv i(p-2)\Mod{p-1}}$. This is equivalent to ${a\equiv (p-1-i)\Mod{p-1}}$.
\end{proof}
For ${0\leq a \leq p-2}$, ${0\leq j\leq p-1}$ and ${1\leq b\leq p-1}$, we define the following functions
\begin{equation}
    \label{psi(a,j) definition}
    \psi(a,j)
        =
        \left\{
            \begin{array}{ll}
                 1&\text{if }a < \ell_{[0:1],j}+1\text{ and }p-1-\ell_{[1:0],j}\leq a \\
                 -1&\text{if }a < p-1-\ell_{[1:0],j}\text{ and }\ell_{[0:1],j}+1\leq a\\
                 0&\text{otherwise}
            \end{array}
        \right.
\end{equation}
\begin{equation}
    \label{sigma(b) definition}
    \sigma(b)
    =
    \left\{
        \begin{array}{ll}
             1&\text{if }(2m+b-1)\equiv 0\Modwb{p}\text{ and }\ell_{[0:1],b-1} \in \{0,1\}  \\
             1&\text{if }(2m+b-1)\not\equiv 0\Modwb{p}\text{ and }\ell_{[0:1],b-1} = 0\\
             0&\text{otherwise}
        \end{array} 
    \right. .
\end{equation}
We now move onto the final bit of notation needed from \cite{bleher}: for ${z \in \Br(\lambda_T)}$, let ${y(z) \in C/U \cong \mathbb{P}^1(\mathbb{F})}$ be such that ${\lambda_T(y(z)) = z}$. For ${0\leq a\leq p-2}$ and ${0\leq j\leq p-1}$, let
\begin{equation}
    \label{n(a,j) definition}
    n(a,j) = n_j + \sum_{z \in \Br(\lambda_T)}
        \left(
            \sum_{d=1}^{\ell_{y(z),j}}\mu_{a,-d}(y(z))
            -
            \sum_{d=1}^{e_{y(z)}-1}\frac{d}{e_{y(z)}}\mu_{a,d}(y(z))
        \right) .
\end{equation}
\begin{lem}
    \label{n(a,j) lemma}
    We have that ${n(a,j)}$ as in (\ref{n(a,j) definition}) is given by
    $$
    {
        n(a,j)
        =
        1
        +
        m
        -
        \ceil*{\frac{m}{p-1}}
        +
        \floor*{\frac{m-j-\ceil*{\frac{2m+j}{p}}}{p-1}}
        +
        \psi(a,j)
    }
    $$
    where ${\psi(a,j)}$ is as given in (\ref{psi(a,j) definition}).
\end{lem}
\begin{proof}
    By Lemma \ref{quotient of p1(F) by T} Lemma \ref{n_j lemma} the claim in the above statement is equivalent to
    $$
    {
        \left(
            \sum_{d=1}^{\ell_{[1:0],j}}\mu_{a,-d}([1:0])
            -
            \sum_{d=1}^{p-2}\frac{d}{p-1}\mu_{a,d}([1:0])
        \right)
        +
        \left(
            \sum_{d=1}^{\ell_{[0:1],j}}\mu_{a,-d}([0:1])
            -
            \sum_{d=1}^{p-2}\frac{d}{p-1}\mu_{a,d}([0:1])
        \right)
        =
        \psi(a,j).
    }
    $$
    From (\ref{mu_{a,i} definition}), we can write
    $$
    {
        \sum_{d=1}^{\ell_{[1:0],j}}\mu_{a,-d}([1:0])
        =
        \sum_{d=1}^{\ell_{[1:0],j}}\mu_{a,p-1-d}([1:0])
        =
        \sum_{d=p-1-\ell_{[1:0],j}}^{p-2}\mu_{a,d}([1:0]).
    }
    $$
    Using (\ref{mu_{a,i} definition}), this is ${1}$ if and only if ${p-1-\ell_{[1:0],j}\leq a\leq p-2}$. For the other sum,
    $$
    {
        \sum_{d=1}^{\ell_{[0:1],j}}\mu_{a,-d}([0:1])
        =
        \sum_{d=1}^{\ell_{[0:1],j}}\mu_{a,p-1-d}([0:1]).
    }
    $$
    Using (\ref{mu_{a,i} definition}), this is ${1}$ if and only if ${1\leq a \leq \ell_{[0:1],j}}$, and $0$ otherwise. Finally, consider the two sums
    $$
    {
        \sum_{d=1}^{p-2}\frac{d}{p-1}\mu_{a,d}([1:0])
        +
        \sum_{d=1}^{p-2}\frac{d}{p-1}\mu_{a,d}([0:1]).
    }
    $$
    The left-hand sum simplifies to ${0}$ if ${a=0}$, and to ${a/(p-1)}$ otherwise. Similarly, the right sum simplifies to $0$ if ${a=0}$ and to ${(p-1-a)/(p-1)}$ otherwise. Therefore, if ${a > 0}$ the two sums add to $1$, and to $0$ otherwise. In total, the four sums can be written as
    $$
    {
        \left\{
            \begin{array}{ll}
                 1&\text{if }p-1-\ell_{[1:0],j}\leq a\leq p-2  \\
                 0&\text{otherwise} 
            \end{array}
        \right.
        +
        \left\{
            \begin{array}{ll}
                 1&\text{if }1\leq a\leq \ell_{[0:1],j} \\
                 0&\text{otherwise} 
            \end{array}
        \right.
        -
        \left\{
            \begin{array}{ll}
                 1&\text{if }a > 0  \\
                 0&\text{otherwise} 
            \end{array}
        \right. ,
    }
    $$
    which indeed simplifies to ${\psi(a,j)}$ after some case taking.
\end{proof}

\begin{thm}
    \label{decomposition as FB module, q=p}
    We have the decomposition
    $$
    {
        \Res^G_B(H^0(C,\Omega_C^{\otimes m}))
        \cong
        \bigoplus_{b=1}^{p}\bigoplus_{a=0}^{p-2}U_{a,b}^{\oplus n_{a,b}}
    }
    $$
    where, for ${1\leq b\leq p-1}$,
    $$
    {
        n_{a,b} =
        \sigma(b)
        +
        \psi(a,b-1) - \psi(a,b)
    }
    $$
    and for ${b=p}$,
    $$
    {
        n_{a,p} =
        m - \ceil*{\frac{m}{p-1}} + \floor*{\frac{m-1-\ceil*{\frac{2m-1}{p}}}{p-1}} + \psi(a,p-1).
    }
    $$
\end{thm}
We note that the case of quadratic differentials (i.e. when ${m=2}$) is explicitly worked out in Corollary \ref{FBconvariants computation}.
\begin{proof}
    By \cite[Thm. 1.6]{bleher},
    $$
    {
        n_{a,b}
        =
        \left\{
            \begin{array}{ll}
                 n(a,j) - n(a,j+1)&\text{if }b=(j+1)p^{|U|-n_I}\text{ with }0\leq j\leq p^{n_I}-2  \\
                 n(a,p^{n_I}-1)&\text{if }b=p^{|U|}\\
                 0&\text{otherwise}
            \end{array}
        \right. .
    }
    $$
    In our case, ${n=1}$ and ${n_I = 1}$. And so
    $$
    {
        n_{a,b}
        =
        \left\{
            \begin{array}{ll}
                 n(a,b-1) - n(a,b)&\text{if }1\leq b\leq p-1  \\
                 n(a,p-1)&\text{if }b=p
            \end{array}
        \right. .
    }
    $$
    Using Lemma \ref{n(a,j) lemma}, for ${1\leq b\leq p-1}$ we get
    $$
    {
        n_{a,b}
        =
        \floor*{\frac{m-b+1-\ceil*{\frac{2m+b-1}{p}}}{p-1}}
        -
        \floor*{\frac{m-b-\ceil*{\frac{2m+b}{p}}}{p-1}}
        +
        \psi(a,b-1) - \psi(a,b).
    }
    $$
    We have that
    $$
    {
        \ceil*{\frac{2m+b}{p}}
        =
        \delta(b) + \ceil*{\frac{2m+b-1}{p}}
    }
    $$
    where ${\delta(b) = 1}$ if ${2m+b-1\equiv 0\Modwb{p}}$, and ${\delta(b) = 0}$ otherwise. Hence,
    $$
    {
        \floor*{\frac{m-b-\ceil*{\frac{2m+b}{p}}}{p-1}}
        =
        \floor*{\frac{m-b-\delta(b)-\ceil*{\frac{2m+b-1}{p}}}{p-1}}
        =
        \floor*{\frac{-\delta(b) - 1}{p-1} + \frac{m-b+1-\ceil*{\frac{2m+b-1}{p}}}{p-1}}.
    }
    $$
    From the definition of ${\ell_{[0:1],b-1}}$, we get
    $$
    {
        \frac{m - b + 1 - \ceil*{\frac{2m+b-1}{p}}}{p-1} = k + \frac{\ell_{[0:1],b-1}}{p-1}\text{ for some }k \in \mathbb{Z}
    }
    $$
    and so
    $$
    {
        \floor*{\frac{m-b+1-\ceil*{\frac{2m+b-1}{p}}}{p-1}}
        -
        \floor*{\frac{m-b-\ceil*{\frac{2m+b}{p}}}{p-1}}
        =
        \floor*{k + \frac{\ell_{[0:1],b-1}}{p-1}}
        -
        \floor*{k + \frac{\ell_{[0:1],b-1} - \delta(b) - 1}{p-1}}.
    }
    $$
    This simplifies to
    $$
    {
        -\floor*{\frac{\ell_{[0:1],b-1} - \delta(b) - 1}{p-1}}.
    }
    $$
    If ${\delta(b) = 1}$, i.e. ${2m+b-1\equiv 0\Modwb{p}}$, then this is $1$ if ${\ell_{[0:1],b-1} \in \{0,1\}}$ and $0$ otherwise. If ${\delta(b) = 0}$, i.e. ${2m+b-1 \not\equiv 0\Modwb{p}}$, then this is $1$ if ${\ell_{[0:1],b-1} = 0}$ and ${0}$ otherwise. From (\ref{sigma(b) definition}), we thus obtain
    $$
    {
        \floor*{\frac{m-b+1-\ceil*{\frac{2m+b-1}{p}}}{p-1}}
        -
        \floor*{\frac{m-b-\ceil*{\frac{2m+b}{p}}}{p-1}}
        =
        \sigma(b)
    }
    $$
    which proves that for ${1\leq b\leq p-1}$, we have ${n_{a,b} = \sigma(b) + \psi(a,b-1) - \psi(a,b)}$ as claimed. For ${n_{a,p}}$, from Lemma \ref{n(a,j) lemma} we get:
    $$
    {
        n_{a,p} =
        1 + m - \ceil*{\frac{m}{p-1}}
        +
        \floor*{\frac{m-(p-1)-\ceil*{\frac{2m+p-1}{p}}}{p-1}}
        +
        \psi(a,p-1)
    }
    $$
    which, after noting
    $$
    {
        \floor*{\frac{m-(p-1)-\ceil*{\frac{2m+p-1}{p}}}{p-1}}
        =
        -1
        +
        \floor*{\frac{m-1-\ceil*{\frac{2m-1}{p}}}{p-1}}
    }
    $$
    gives us the desired result.
\end{proof}
When ${p > 3m}$, Corollary \ref{final decomposition when p > 3m for B} below will simplify our result in Theorem \ref{decomposition as FB module, q=p}. To derive this simplification we need the following observations.
\begin{lem}
    \label{summing up the psis}
    For ${0\leq j\leq p-1}$, we have the following equality:
    $$
    {
        \sum_{a=0}^{p-2}\psi(a,j) = \ell_{[0:1],j} + \ell_{[1:0],j} - (p-2).
    }
    $$
\end{lem}
\begin{proof}
    We have
    $$
    {
        \sum_{a=0}^{p-2}\psi(a,j)
        =
        \sum_{a=0}^{p-2-\ell_{[1:0],j}}\psi(a,j) + \sum_{a=p-1-\ell_{[1:0],j}}^{p-2}\psi(a,j).
    }
    $$
    Note that in the region ${0\leq a\leq p-2-\ell_{[1:0],j}}$ we have that ${\psi(a,j)\leq 0}$. For ${p-1-\ell_{[1:0],j}\leq a\leq p-2}$, we have ${\psi(a,j) \geq 0}$. Furthermore, from the definition of ${\psi(a,j)}$, we have
    \begin{equation*}
        \begin{aligned}
            \sum_{a=0}^{p-2-\ell_{[1:0],j}}\psi(a,j)
        &=
        \left\{
            \begin{array}{ll}
                 -(p-2-\ell_{[1:0],j} - (\ell_{[0:1],j} + 1) + 1)&\text{if }\ell_{[0:1],j} + 1\leq p-2-\ell_{[1:0],j}  \\
                 0&\text{if }\ell_{[0:1],j} + 1\geq p-1-\ell_{[1:0],j}
            \end{array}
        \right.\\
        &=\left\{
            \begin{array}{ll}
                 \ell_{[0:1],j} + \ell_{[1:0],j} - (p-2)&\text{if }\ell_{[0:1],j} + 1\leq p-2-\ell_{[1:0],j}  \\
                 0&\text{if }\ell_{[0:1],j} + 1\geq p-1-\ell_{[1:0],j}
            \end{array}
        \right.
        \end{aligned}
    \end{equation*}
    and also
    \begin{equation*}
        \begin{aligned}
            \sum_{a=p-1-\ell_{[1:0],j}}^{p-2}\psi(a,j)
        &=
        \left\{
            \begin{array}{ll}
                 \ell_{[0:1],j} - (p - 1 - \ell_{[1:0],j}) + 1&\text{if }p-1-\ell_{[1:0],j}\leq \ell_{[0:1],j}  \\
                 0&\text{if }p-1-\ell_{[1:0],j}\geq \ell_{[0:1],j}+1
            \end{array}
        \right.\\
        &=\left\{
            \begin{array}{ll}
                 \ell_{[0:1],j} + \ell_{[1:0],j} - (p-2)&\text{if }p-1-\ell_{[1:0],j}\leq \ell_{[0:1],j}  \\
                 0&\text{if }p-1-\ell_{[1:0],j}\geq \ell_{[0:1],j}+1
            \end{array}
        \right. .
        \end{aligned}
    \end{equation*}
    We show that the sum of these sums is equal to ${\ell_{[0:1],j} + \ell_{[1:0],j} - (p-2)}$. We do this by taking cases on the inequalities. In two of the four resulting cases, this is obvious. Note that we cannot have both
    $$
    {
        \ell_{[0:1],j}+1\leq p-2-\ell_{[1:0],j}
        \text{ and }
        p-1-\ell_{[1:0],j}
        \leq \ell_{[0:1],j} 
    }
    $$
    because this would imply ${p-\ell_{[1:0],j}\leq p-2-\ell_{[1:0],j}}$. Finally, if
    $$
    {
        \ell_{[0:1],j}+1\geq p-1-\ell_{[1:0],j}
        \text{ and }
        p-1-\ell_{[1:0],j}\geq\ell_{[0:1],j}+1
    }
    $$
    this implies that ${p-1-\ell_{[1:0],j} = \ell_{[0:1],j}+1}$, and hence implies ${\ell_{[1:0],j} + \ell_{[0:1],j} = p-2}$. In this case, the individual sums are both $0$, and the total can still be written as ${\ell_{[1:0],j} + \ell_{[0:1],j} - (p-2)}$.
\end{proof}
\begin{cor}
    \label{n_b is either 1 or 2}
    For ${1\leq b\leq p}$, let
    $$
    {
        n_b =
        \sum_{a=0}^{p-2}n_{a,b}.
    }
    $$
    In other words, ${n_b}$ counts how many indecomposables we have of dimension ${b}$. Then we have
    $$
    {
        n_b =
        \left\{
            \begin{array}{ll}
                 2&\text{if }2m+b-1\equiv 0\Modwb{p}  \\
                 1&\text{if }2m+b-1\not\equiv 0\Modwb{p} 
            \end{array}
        \right. .
    }
    $$
\end{cor}
\begin{proof}
    For this, we use ${n_{a,b} = \sigma(b) + \psi(a,b-1) - \psi(a,b)}$ (Theorem \ref{decomposition as FB module, q=p}) and take cases. Note that by Lemma \ref{summing up the psis}, we have
    $$
    {
        n_b = (p-1)\sigma(b) + \left(\ell_{[0:1],b-1} + \ell_{[1:0],{b-1}} - (p-2)\right)
        -
        \left(\ell_{[0:1],b} + \ell_{[1:0],b} - (p-2)\right)
    }
    $$
    which simplifies to
    $$
    {
        n_b = (p-1)\sigma(b) + \ell_{[0:1],b-1} - \ell_{[0:1],b}.
    }
    $$
    Firstly, note that if ${2m+b-1 \equiv 0 \Modwb{p}}$, then
    $$
    {
        \begin{array}{lll}
             \ell_{[0:1],b}&\equiv&\left(m-b-\ceil*{\frac{2m+b}{p}}\right)\Mod{p-1}  \\
             &\equiv&\left(m-b+1-\ceil*{\frac{2m+b-1}{p}} - 2\right)\Mod{p-1} \\
             &\equiv&\left(\ell_{[0:1],b-1} - 2\right)\Mod{p-1}
        \end{array} .
    }
    $$
    If ${2m+b-1\not\equiv 0\Modwb{p}}$, then we get a similar relation: ${\ell_{[0:1],b} \equiv \left(\ell_{[0:1],b-1} - 1\right)\Mod{p-1}}$. Enumerating the possible cases, we get the following table:
    \begin{table}[H]
\centering
\begin{tabular}{|l|l|l|l|l|}
\hline
\rowcolor[HTML]{9B9B9B}
${2m+b-1\equiv 0 \Mod{p}}$ & ${\ell_{[0:1],{b-1}}}$ & ${\ell_{[0:1],b}}$       & ${\sigma(b)}$ & ${n_b = (p-1)\sigma(b) + \ell_{[0:1],b-1} - \ell_{[0:1],b}}$ \\ \hline
True                       & $0$                    & ${p-3}$                  & $1$           & ${(p-1) - (p-3) = 2}$                                        \\ \hline
True                       & $1$                    & $p-2$                      & $1$           & ${(p-1) + 1 - (p-2) = 2}$                                    \\ \hline
True                       & ${\geq 2}$             & ${\ell_{[0:1],b-1} - 2}$ & $0$           & ${0 + \ell_{[0:1],b-1} - (\ell_{[0:1],b-1} - 2) = 2}$        \\ \hline
False                      & $0$                    & ${p-2}$                  & $1$           & ${(p-1) + 0 - (p-2) = 1}$                                    \\ \hline
False                      & ${\geq 1}$             & ${\ell_{[0:1],b-1} - 1}$ & $0$           & ${\ell_{[0:1],b-1} - (\ell_{[0:1],b-1} - 1) = 1}$            \\ \hline
\end{tabular}
\end{table}
    This concludes the proof of the corollary.
\end{proof}
\begin{cor}
    \label{each n_{a,b} is either 0 or 1}
    Let ${1\leq b\leq p-1}$. If ${2m+b-1\not\equiv 0\Modwb{p}}$, there is exactly one value ${0\leq a_*\leq p-2}$ such that ${n_{a_*,b} = 1}$, with all other ${n_{a,b} = 0}$. If ${2m+b-1 \equiv 0\Modwb{p}}$, there are exactly two different \\${0\leq a_*,a_*'\leq p-2}$ such that ${n_{a_*,b} = n_{a_*',b} = 1}$, with all other ${n_{a,b} = 0}$.
\end{cor}
\begin{proof}
    If ${2m+b-1\not\equiv 0\Modwb{p}}$, then this follows immediately from Corollary \ref{n_b is either 1 or 2}, since each ${n_{a,b}}$ is a non-negative integer. Suppose that ${2m+b-1 \equiv 0\Modwb{p}}$. Then by Corollary \ref{n_b is either 1 or 2} it suffices to show that there does not exist ${0\leq a\leq p-2}$ with ${n_{a,b} = 2}$. It is not possible to have ${\psi(a,b-1) - \psi(a,b) = 2}$. To see this, note that since ${\psi(a,j) \in \{-1,0,1\}}$, this would imply ${\psi(a,b-1) = 1}$ and ${\psi(a,b) = -1}$. This is impossible from the definition of ${\psi(a,j)}$, because of the condition imposed on ${\ell_{[1:0],b-1} = \ell_{[1:0],b}}$. Hence, since ${n_{a,b} = \sigma(b) + \psi(a,b-1) - \psi(a,b)}$ (Theorem \ref{decomposition as FB module, q=p}), we conclude ${n_{a,b} = 2}$ if and only if ${\sigma(b) = 1}$ and ${\psi(a,b-1) - \psi(a,b) = 1}$. However, if ${\sigma(b) = 1}$, note that ${n_{0,b} = \sigma(b) + \psi(0,b-1) - \psi(0,b) = 1 + 0-0 = 1}$. Thus we obtain ${n_b = n_{a,b} + n_{0,b} = 2 + 1 = 3}$, which contradicts Corollary \ref{n_b is either 1 or 2}.
\end{proof}
One may now hope that given ${1\leq b\leq p-1}$, it might be possible to explicitly find the values of $a$ with ${n_{a,b} = 1}$. Unfortunately, in the general case for any prime $p$ and ${m \geq 2}$, it is very hard to predict the exact outcome of the ${n_{a,b}}$ values. But if $p$ is taken to be large enough relative to $m$, the situation does stabilise:
\begin{cor}
    \label{final decomposition when p > 3m for B}
    When ${p}$ is a prime satisfying ${p > 3m}$, we get the following decomposition of ${H^0(C,\Omega_C^{\otimes m})}$ as an ${\mathbb{F}[B]}$ module:
    \begin{align*}
        \Res^G_B(H^0(C,\Omega_C^{\otimes m}))
        &\cong \bigoplus_{b=1}^{m}U_{m-b,b} 
            \oplus
            \bigoplus_{b=m+1}^{p-2m+1}U_{p-1+m-b,b}
            \oplus
            \bigoplus_{b=p-2m+1}^{p-1}U_{p-2+m-b,b}\\
            &\oplus
            \bigoplus_{a=0,\ a\neq m-1}^{p-2}U_{a,p}^{\oplus (m-1)}
            \oplus
            U_{m-1,p}^{\oplus (m-2)}.
    \end{align*}
\end{cor}
\begin{proof}
    For ${0\leq j\leq p-1}$, we first compute ${\ell_{[0:1],j}}$ and ${\ell_{[1:0],j}}$. We have
    $$
    {
        \ell_{[1:0],j} \equiv m(p-2)\Mod{p-1} \equiv p-1-m\Mod{p-1}.
    }
    $$
    Since ${p > 3m}$, we obtain ${0\leq p-1-m\leq p-2}$, and hence ${\ell_{[1:0],j} = p-1-m}$. Next,
    $$
    {
        \ell_{[0:1],j}
        \equiv
        \left(
            m-j-\ceil*{\frac{2m+j}{p}}
        \right)\Mod{p-1}
        \equiv
        \left\{
            \begin{array}{ll}
                 m-j-1&0\leq j\leq p-2m  \\
                 m-j-2&p-2m+1\leq j\leq p-1 
            \end{array}
        \right.
        \Mod{p-1}
    }
    $$
    and so
    $$
    {
        \ell_{[0:1],j}
        =
        \left\{
            \begin{array}{ll}
              m-1-j&0\leq j\leq m-1\\
              p - 2 + m-j&m\leq j\leq p-2m\\
             p-3 + m-j&p-2m+1\leq j\leq p-1
            \end{array}
        \right. .
    }
    $$
    Note that the bounds on $j$ above are non-empty because (1) ${m < p-2m}$ from our assumption that ${p > 3m}$, and (2) ${p-2m + 1 < p-1}$ is equivalent to ${-2m + 2 < 0}$, which is always true when ${m\geq 2}$. We also check that the expressions given are between ${0}$ and ${p-2}$ (inclusive):
\begin{table}[H]
\begin{tabular}{|l|l|l|l|}
\hline
\rowcolor[HTML]{9B9B9B}
Expression    & Range of $j$             & Maximal value of expression & Minimal value of expression \\ \hline
${m-1-j}$     & ${0\leq j\leq m-1}$      & ${m-1}$                         & ${0}$                           \\ \hline
${p-2 + m-j}$ & ${m\leq j\leq p-2m}$     & ${p-2}$                         & ${3m-2}$                        \\ \hline
${p-3 + m-j}$ & ${p-2m+1\leq j\leq p-1}$ & ${3m-4}$                        & ${m-2}$                         \\ \hline
\end{tabular}
\end{table}
    Because of our assumption ${p > 3m}$ the minimal and maximal values all lie between $0$ and ${p-2}$ inclusive, and thus it must be the case the ${\ell_{[0:1],j}}$ do as well for all $j$ within their respective ranges.\\
    \\
    As stated in Corollary \ref{each n_{a,b} is either 0 or 1}, for ${1\leq b\leq p-1}$, if ${2m + b-1 \not\equiv 0\Modwb{p}}$ there is only one value ${0\leq a\leq p-2}$ such that ${n_{a,b} = 1}$, and two values if ${2m + b-1 \equiv 0 \Modwb{p}}$. As ${p > 3m}$, ${2m+b-1 \equiv 0 \Modwb{p}}$ if and only if ${2m + b-1 = p}$, i.e. ${b = p - 2m + 1}$. We show these values of $a$ are given by
    $$
    {
        1 =
        \left\{
        \begin{array}{ll}
             n_{m-b,b}&\text{for }1\leq b\leq m  \\
             n_{p-1+m-b,b}&\text{for }m+1\leq b\leq p-2m\\
             n_{3m-3,p-2m+1}&\\
             n_{3m-2,p-2m+1}&\\
             n_{p-2+m-b,b}&\text{for }p-2m+2\leq b\leq p-1\\
        \end{array}
        \right.
    }
    $$
    i.e.
    $$
    {
        1
        =
        \left\{
        \begin{array}{ll}
             n_{m-b,b}&\text{for }1\leq b\leq m  \\
             n_{p-1+m-b,b}&\text{for }m+1\leq b\leq p-2m+1\\
             n_{p-2+m-b,b}&\text{for }p-2m+1\leq b\leq p-1\\
        \end{array}
        \right. .
    }
    $$
    Recall by Theorem \ref{decomposition as FB module, q=p}, ${n_{a,b} = \sigma(b) + \psi(a,b-1) - \psi(a,b)}$ (where ${\psi}$,${\sigma}$ are defined in eq. (\ref{psi(a,j) definition}),(\ref{sigma(b) definition}) respectively). For ${1\leq b\leq p-1}$, we have
    $$
    {
        \sigma(b) =
        \left\{
            \begin{array}{ll}
                 1&\text{if }b=m  \\
                 0&\text{otherwise} 
            \end{array}
        \right. .
    }
    $$
    To see this, note that (1) ${\ell_{[0:1],b-1} = 0}$ if and only if ${b = m}$, and (2) ${\ell_{[0:1],p-2m} = 3m-2 \notin \{0,1\}}$ for ${m\geq 2}$.
    Hence for ${b\neq m}$, it suffices to verify ${\psi(a,b-1) - \psi(a,b) = 1}$ for the given values of $a$. Firstly, in the range ${1\leq b\leq m-1}$ we have
    \begin{align*}
        &\psi(m-b,b-1) - \psi(m-b,b)\\
        =&\left\{
            \begin{array}{ll}
                 1&\text{if }m-b < m-b+1\text{ and }m\leq m-b\leq p-2 \\
                 -1&\text{if }m-b < m\text{ and }m-b+1\leq m-b\leq p-2\\
                 0&\text{otherwise}
            \end{array}
        \right.
        \\
        &-\left\{
            \begin{array}{ll}
                 1&\text{if }m-b < m-b\text{ and }m\leq a\leq p-2 \\
                 -1&\text{if }m-b < m\text{ and }m-b\leq m-b\leq p-2\\
                 0&\text{otherwise}
            \end{array}
        \right. \\
        =&\ \ 0-(-1)=1
    \end{align*}
    as claimed. When ${b=m}$, ${\sigma(m) = 1}$ and
    \begin{align*}
        &\psi(m,m-1) - \psi(m,m)\\
        =&\left\{
            \begin{array}{ll}
                 1&\text{if }m < 1\text{ and }m\leq m\leq p-2 \\
                 -1&\text{if }m < m\text{ and }1\leq m\leq p-2\\
                 0&\text{otherwise}
            \end{array}
        \right.
        -\left\{
            \begin{array}{ll}
                 1&\text{if }m < p-1\text{ and }m\leq m\leq p-2 \\
                 -1&\text{if }m < m\text{ and }p-1\leq m\leq p-2\\
                 0&\text{otherwise}
            \end{array}
        \right.\\
        =&\ \ \ 0-0=0.
    \end{align*}
    Hence ${n_{0,m} = 1}$. Indeed, this $a$ value satisfies ${a = m-b}$. Next, we consider the range ${m+1\leq b\leq p-2m}$. Note that ${0\leq p-1 + m-b\leq p-2}$ in this region, and
    \begin{align*}
        &\psi(p-1 + m-b,b-1) - \psi(p-1 + m-b,b)\\
        =&\left\{
            \begin{array}{ll}
                 1&\text{if }p - 1 + m-b < p + m - b\text{ and }m\leq p-1 + m-b\leq p-2 \\
                 -1&\text{if }p-1 + m-b < m\text{ and }p + m-b\leq p-1 + m-b\leq p-2\\
                 0&\text{otherwise}
            \end{array}
        \right.\\
        &-\left\{
            \begin{array}{ll}
                 1&\text{if }p - 1 + m-b < p - 1 + m - b\text{ and }m\leq p-1 + m-b\leq p-2 \\
                 -1&\text{if }p-1 + m-b < m\text{ and }p-1 + m-b\leq p-1 + m-b\leq p-2\\
                 0&\text{otherwise}
            \end{array}
        \right.\\
        =&\ \ \ 1-0=1
    \end{align*}
    as claimed. We now consider the case ${b = p-2m+1}$ separately, since this is the unique value of $b$ for which we must find two values for $a$. Note that ${0\leq 3m-3\leq p-2}$ and ${0\leq 3m-2\leq p-2}$. Furthermore, we have
    \begin{align*}
        &\psi(a,p-2m) - \psi(a,p-2m + 1)\\
        =&\left\{
            \begin{array}{ll}
                 1&\text{if }a < 3m-1\text{ and }m\leq a\leq p-2 \\
                 -1&\text{if }a < m\text{ and }3m-1\leq a\leq p-2\\
                 0&\text{otherwise}
            \end{array}
        \right.
        -
        \left\{
            \begin{array}{ll}
                 1&\text{if }a < 3m-3\text{ and }m\leq a\leq p-2 \\
                 -1&\text{if }a < m\text{ and }3m-3\leq a\leq p-2\\
                 0&\text{otherwise}
            \end{array}
        \right. .
    \end{align*}
    The above expression is $1$ when ${a \in \{3m-3,3m-2\}}$, as claimed. Next, we consider the range\\ ${p-2m+2\leq b\leq p-2}$. We get ${0\leq p-2 + m-b\leq p-2}$, and
    \begin{align*}
        &\psi(p-2 + m-b,b-1) - \psi(p-2 + m-b,b)\\
        =&\left\{
            \begin{array}{ll}
                 1&\text{if }p-2 + m-b < p-1+m-b\text{ and }m\leq p-2 + m-b\leq p-2\\
                 -1&\text{if }p-2 + m-b < m\text{ and }p-1+m-b\leq p-2 + m-b\leq p-2 \\
                 0&\text{ otherwise}
            \end{array}
        \right.\\
        &-\left\{
            \begin{array}{ll}
                 1&\text{if }p-2 + m-b < p-2+m-b\text{ and }m\leq p-2 + m-b\leq p-2\\
                 -1&\text{if }p-2 + m-b < m\text{ and }p-2+m-b\leq p-2 + m-b\leq p-2 \\
                 0&\text{otherwise}
            \end{array}
        \right.\\
        =&\ \ \ 1-0=1
    \end{align*}
    as claimed. When ${b=p-1}$, note that
    \begin{align*}
        &\psi(m-1,p-2) - \psi(m-1,p-1)\\
        =&\left\{
            \begin{array}{ll}
                 1&\text{if }m-1 < m\text{ and }m\leq m-1\leq p-2  \\
                 -1&\text{if }m-1 < m\text{ and }m\leq m-1\leq p-2 \\
                 0&\text{ otherwise}
            \end{array}
        \right.
        -
        \left\{
            \begin{array}{ll}
                 1&\text{if }m-1 < m-1\text{ and }m\leq m-1\leq p-2  \\
                 -1&\text{if }m-1 < m\text{ and }m-1\leq m-1\leq p-2 \\
                 0&\text{otherwise}
            \end{array}
        \right.\\
        =&\ \ \ 0-(-1) = 1 .
    \end{align*}
    This $a$ value indeed satisfies ${a=p-2 + m-b}$.\\
    Finally, we consider ${b=p}$. Note that using Theorem \ref{decomposition as FB module, q=p} with the assumption ${p > 3m}$ yields
    \begin{align*}
        n_{a,p} &= m - \ceil*{\frac{m}{p-1}} + \floor*{\frac{m-1-\ceil*{\frac{2m-1}{p}}}{p-1}} + \psi(a,p-1)\\
        &=m-1 + \floor*{\frac{m-1-1}{p-1}} + \psi(a,p-1)\\
        &=m-1 + \psi(a,p-1)\\
        &=m-1 + \left\{
            \begin{array}{ll}
                 1&\text{if }a < m-1\text{ and }m\leq a  \\
                 -1&\text{if }a < m\text{ and }m-1\leq a \\
                 0&\text{ otherwise}
            \end{array}
                \right.\\
        &= \left\{
                \begin{array}{ll}
                     m-1&\text{if }a\neq m-1  \\
                     m-2&\text{if }a = m-1 
                \end{array}
            \right. .
    \end{align*}
    This concludes the proof of the corollary.
\end{proof}
We finish off this section by computing  the dimension of the space of $B$-coinvariants ${\dim_{\mathbb{F}}H^0(C,\Omega_C^{\otimes 2})_B}$. As explained in \cite[\S 3]{bleher}, this has applications to deformation theory.
\begin{cor}
    \label{FBconvariants computation}
    We have
    $$
    {
        \dim_{\mathbb{F}}H^0(C,\Omega_C^{\otimes 2})_B
        =
        \left\{
            \begin{array}{ll}
                 2&\text{if }p=3  \\
                 1&\text{if }p\geq 5  \\
            \end{array}
        \right. .
    }
    $$
\end{cor}
\begin{proof}
    Taking coinvariants is additive. Furthermore, since the ${\mathbb{F}[B]}$ modules ${U_{a,b}}$ are uniserial, ${(U_{a,b})_B \neq 0}$ if and only if ${U_{a,b}/\Rad(U_{a,b}) \cong S_0}$. Using Proposition \ref{parameterisation of F[B] modules}, the top composition factor ${U_{a,b}/\Rad(U_{a,b})}$ of ${U_{a,b}}$ is given by ${S_{a + 2(b-1)}}$. Hence, ${(U_{a,b})_B \neq 0}$ if and only if ${a + 2(b-1) \equiv 0 \Mod{p-1}}$. When ${p > 5}$, by Corollary \ref{final decomposition when p > 3m for B}
    \begin{align*}
        \Res^G_B(H^0(C,\Omega_C^{\otimes 2}))
        &\cong \bigoplus_{b=1}^{2}U_{2-b,b} 
            \oplus
            \bigoplus_{b=3}^{p-3}U_{p+1-b,b}
            \oplus
            \bigoplus_{b=p-3}^{p-1}U_{p-b,b}
            \oplus
            \bigoplus_{a=0,\ a\neq 1}^{p-2}U_{a,p}.
    \end{align*}
    Note that
    \begin{align*}
        b \in [1,2]:&\; 2-b + 2(b-1) = b\not\equiv 0\Mod{p-1}\\
        b \in [3,p-3]:&\; p+1-b + 2(b-1) = p-1 + b\not\equiv 0\Mod{p-1}\\
        b \in [p-3,p-1]:&\; p-b + 2(b-1) = p-2+b \equiv b-1\Mod{p-1}\not\equiv 0\Mod{p-1}\\
        a \in [0,p-2],\ a\neq 1:&\; a + 2(p-1) \equiv a\Mod{p-1}.
    \end{align*}
    Hence ${\dim_{\mathbb{F}}H^0(C,\Omega_C^{\otimes 2})_B = 1}$ when ${p > 5}$. When ${p=5}$, it can be verified using Theorem \ref{decomposition as FB module, q=p} that
    $$
    {
        \Res^G_B(H^0(C,\Omega_C^{\otimes 2})) \cong
        U_{0,2} \oplus U_{0,5} \oplus U_{1,1} \oplus
        U_{1,4} \oplus U_{2,3} \oplus U_{2,5} \oplus
        U_{3,2} \oplus U_{3,5}.
    }
    $$
    Hence ${\dim_{\mathbb{F}}H^0(C,\Omega_C^{\otimes 2})_B = 1}$ also when ${p=5}$. Finally, when ${p=3}$, it can be verified using Theorem \ref{decomposition as FB module, q=p} that
    $$
    {
        \Res^G_B(H^0(C,\Omega_C^{\otimes 2})) \cong
        U_{0,1} \oplus U_{0,3} \oplus U_{1,2}.
    }
    $$
    Hence ${\dim_{\mathbb{F}}H^0(C,\Omega_C^{\otimes 2})_B = 2}$ when ${p=3}$.
\end{proof}

\section{Composition factors of ${H^0(C,\Omega_C^{\otimes m})}$ as an ${\mathbb{F}[G]}$ module}
In this section, we compute the composition factors of ${H^0(C,\Omega_C^{\otimes m})}$ as an ${\mathbb{F}[G]}$ module. As explained in Remark \ref{F[G] module decomposition of H^0}, the ${\mathbb{F}[G]}$ module decomposition of ${H^0(C,\Omega_C^{\otimes m})}$ is then a direct corollary of Theorem \ref{decomposition as FB module, q=p} and Theorem \ref{lifting decomposition}.\\
\\
We derive our result using two different methods. The first method utilises results about the equivariant Euler characteristic from \cite{Kock_2004}. The second method provides a composition series using our explicit basis (Proposition \ref{basis for holomorphic polydifferentials}), thus yielding additional structural information. For ${0\leq t\leq p-1}$, let
\begin{equation}
    \label{sigma_t}
    \sigma_t = m - \ceil*{\frac{m+t}{p-1}}.
\end{equation}
\begin{thm}
    \label{composition factors of H^0}
    Both ${V_1}$ and ${V_p}$ occur as a composition factor of ${H^0(C,\Omega_{C}^{\otimes m})}$ with multiplicity ${1 + \sigma_{p-1}}$. For ${2\leq i\leq p-1}$, ${V_i}$ occurs as a composition factor of ${H^0(C,\Omega_C^{\otimes m})}$ with multiplicity ${1 + \sigma_{i-1} + \sigma_{p-i}}$.
\end{thm}
Before giving the first proof of this Theorem, we recall the lower ramification groups for the action of $G$ on $C$:
\begin{prop}
    \label{ramification information for P=[1:0:0]}
    Let ${q=p}$, and let ${P = [1:0:0] \in C}$. Then the lower ramification groups at $P$ are given by:
    $$
    {
        G_{P,i}
        =
        \left\{
            \begin{array}{ll}
                 B&\text{if }i \in \{-1,0\}  \\
                 U&\text{if }i \in [1,p+1]  \\
                 1&\text{if }i \geq p+2
            \end{array}
        \right. .
    }
    $$
\end{prop}
\begin{proof}
    The fact that the lower ramification groups are as given in the Proposition statement follows from:
    \begin{itemize}
        \item ${G_{P,0} = G_{P} = B}$. This is given in \cite[\S 1.2, Lem. 1.25]{lucas-mphil} (this can also be seen directly).
        \item ${U\subseteq G_{P,i}}$ if and only if ${i\leq p+1}$. This is given in \cite[\S 2.1, Lem. 2.6]{lucas-mphil}.
        \item ${G_{P,0}/G_{P,1}}$ is a cyclic group of order prime to $p$, and for ${i\geq 1}$ the quotient ${G_{P,i}/G_{P,i+1}}$ is a product of cyclic groups of order $p$. This follows from general theory about ramification groups (see e.g. \cite[\S \RN{4}]{serre}).
    \end{itemize}
\end{proof}
\begin{rem}
    In fact, Proposition \ref{ramification information for P=[1:0:0]} holds for arbitrary ${q}$.
\end{rem}
\begin{proof}[Proof of Theorem \ref{composition factors of H^0}]
     If ${d_i}$ is the multiplicity of ${V_i}$ as a composition factor of ${H^0(C,\Omega_C^{\otimes m})}$, then
    $$
    {
        [H^0(C,\Omega_C^{\otimes m})] = \sum_{i=1}^{p}d_i[V_i] \text{ in }K_0(G,\mathbb{F})_\mathbb{Q}
    }
    $$
    where ${K_0(G,\mathbb{F})_\mathbb{Q}}$ denotes the rationalised Grothendieck group, with basis given by simple ${\mathbb{F}[G]}$ modules. The equivariant Euler characteristic of the sheaf of polydifferentials is given by
    $$
    {
        \chi(G,C,\Omega_{C}^{\otimes m}) =
        [H^0(C,\Omega_C^{\otimes m})] -
        [H^1(C,\Omega_C^{\otimes m})] \in K_0(G,\mathbb{F})_{\mathbb{Q}}.
    }
    $$
    Recall that the genus of $C$ is ${\frac{p(p-1)}{2}}$ (see Proposition \ref{genus of C}). Furthermore, ${p\geq 3}$ (see \S \ref{notation-section}). Thus, \\${\deg(\Omega_C^{\otimes m}) = m(2g(C)-2) > 2g(C)-2}$, from which it follows that ${H^1(C,\Omega_C^{\otimes m}) = 0}$ (see \cite[\S \RN{4}, Ex. 1.3.3]{hartshorne} and \cite[\S \RN{4}, Ex. 1.3.4]{hartshorne}). The equivariant Euler characteristic is therefore equal to ${[H^0(C,\Omega_C^{\otimes m})]}$. Finding the ${d_i}$ thus comes down to computing ${\chi(G,C,\Omega_C^{\otimes m})}$.\\
    \\
    By \cite[\S 3, Thm. 3.1]{Kock_2004} and using the fact that ${\Omega_C(P)}$ is given by the fundamental character ${\theta_P}$ at $P$ and that ${\deg(\Omega_C) = 2g(C)-2}$, we can write the equivariant Euler characteristic as
    \begin{equation}
        \label{Theorem 3.1 bernhard}
        \chi(G,C,\Omega_{C}^{\otimes m})
        =
        c[\mathbb{F}[G]]
        -
        \frac{1}{|G|}
        \sum_{P \in C}e_P^w
        \sum_{d=0}^{e_{P}^t-1}d \cdot
        [\Ind_{G_P}^{G}(\theta_P^{\otimes (m+d)})]
        \in K_0(G,\mathbb{F})_{\mathbb{Q}}
    \end{equation}
    where
    \begin{equation}
        \label{Theorem 3.1 constant c}
        c =
        (1-g_Y)
        + \frac{m(2g-2)}{|G|}
        - \frac{1}{2|G|}
        \sum_{P \in C}\left(
        (e_P^w-1)(e_P^t + 1)
        +
        \sum_{s\geq 2}(|G_{P,s}|-1)
        \right)
    \end{equation}
    and
    \begin{itemize}
        \item ${\theta_P}$ is the fundamental character of ${G_P}$ at $P$.
        \item ${g_Y}$ is the genus of ${X/G \cong \mathbb{P}^1}$, which is $0$.
        \item ${g = g(C)}$ is the genus of $C$, which for our curve is ${g(C) = \frac{p(p-1)}{2}}$ (see Proposition \ref{genus of C}).
        \item ${e_P^w}$ is ${|G_{P,1}|}$.
        \item ${e_P^t}$ is ${|G_P/G_{P,1}|}$.
    \end{itemize}
    Recall from \cite[\S 1.2, Lem. 1.28]{lucas-mphil}, the points of $C$ ramified under the action of $G$ are the points at infinity ${[X:Y:0]}$. Furthermore, these points are all in the same $G$-orbit. In particular, if ${Q = [X:Y:0]}$, ${G_{Q,i}}$ is conjugate to ${G_{P,i}}$ for all ${i\geq 0}$. By elementary representation theory, the fundamental representation ${\theta_Q}$ at $Q$ is obtained from the fundamental representation ${\theta_P}$ at $P$ via conjugation, and ${\Ind_{G_Q}^G(\theta_Q) \cong \Ind_{G_P}^G(\theta_P)}$. The fundamental representation at ${P = [1:0:0]}$ is ${S_{p-2}}$. To see this, note by Proposition \ref{ramification information for P=[1:0:0]} we need to compute the action of ${G_P = B}$ on ${\mathfrak{m}_{C,P}/\mathfrak{m}_{C,P}^2}$. The subgroup $U$ acts trivially. It remains to compute the action of $T$. A local parameter at ${P}$ is ${t = Z/X}$ (\cite[\S 1.2, Lem. 1.26]{lucas-mphil}). Note that \\${\left(
        \begin{array}{ll}
             \zeta&0  \\
             0&\zeta^{-1} 
        \end{array}
    \right)\cdot t = Z/(\zeta X) = \zeta^{-1}t = \zeta^{p-2}t}$, and so by definition of ${S_{p-2}}$ in Proposition \ref{parameterisation of F[B] modules}, the fundamental representation at ${P}$ is ${S_{p-2}}$. Using the fact that all the ramified points are in the same orbit as ${P = [1:0:0]}$ and using Proposition \ref{ramification information for P=[1:0:0]} we obtain
    \begin{align*}
        \frac{1}{|G|}
        \sum_{P \in C}e_P^w
        \sum_{d=0}^{e_{P}^t-1}d \cdot
        \left[\Ind_{B}^{G}(\theta_P^{\otimes (m+d)})\right]
        &=
        \frac{1}{p(p^2-1)}p(p+1)
        \sum_{d=0}^{p-2}d \cdot
        \left[\Ind_{B}^{G}(S_{(p-2)(m+d)})\right]\\
        &=\frac{1}{p-1}
        \sum_{d=0}^{p-2}d \cdot
        \left[\Ind_{B}^{G}(S_{(p-2)(m+d)})\right].
    \end{align*}
    Once again using Proposition \ref{ramification information for P=[1:0:0]} and the fact all points in $C$ ramified under $G$ are in the same orbit, we get
    $$
    {
        \sum_{P \in C}\left(
        (e_P^w-1)(e_P^t + 1)
        +
        \sum_{s\geq 2}(|G_{P,s}|-1)
        \right)
        =
        (p+1)
        \Big(
            (p-1)
            p
            +
            p
            (p-1)
        \Big)
        =
        2p(p^2-1).
    }
    $$
    Using ${|G| = p(p^2-1)}$ we hence obtain
    $$
    {
        c
        =
        1 +
        \frac{m(p(p-1) - 2)}{p(p^2-1)}
        -
        \frac{p(p^2-1)}{p(p^2-1)}
        =
        \frac{m(p(p-1) - 2)}{p(p^2-1)}
        =
        \frac{m}{p+1}
        -
        \frac{2m}{p(p^2-1)}.
    }
    $$
    Hence overall, we obtain
    $$
    {
        \chi(G,C,\Omega_{C}^{\otimes m})
        =
        \left(\frac{m}{p+1} - \frac{2m}{p(p^2-1)}\right)\left[\mathbb{F}[G]\right]
        -
        \frac{1}{p-1}
        \sum_{d=0}^{p-2}d \cdot
        \left[\Ind_{B}^{G}(S_{(p-2)(m+d)})\right].
    }
    $$
    In what follows, to obtain the composition factors of the regular representation ${\mathbb{F}[G]}$, we make repeated use of the fact that ${\mathbb{F}[G] \cong \bigoplus_{i=1}^{p}P_{V_i}^{\oplus i}}$ (see \cite[\S \RN{2}.5, Cor. 4]{localrep}) and Proposition \ref{projective indecomposable F[G] modules}.\\
    \\
    We first compute the multiplicity ${d_1}$. If ${p=3}$, the only indecomposable projective module containing ${V_1}$ as a composition factor is ${P_{V_1}}$, with multiplicitly ${3}$. Hence ${V_1}$ occurs $3$ times as a composition factor of ${\mathbb{F}[G]}$. If ${p > 3}$, there are two projective indecomposable modules which have ${V_1}$ as a composition factor: ${P_{V_1}}$ and ${P_{V_{p-2}}}$. In ${P_{V_1}}$ it occurs twice, in ${P_{V_{p-2}}}$ it occurs once, and therefore occurs ${2 + 1\cdot (p-2) = p}$ times in ${\mathbb{F}[G]}$. In general, we obtain that ${V_1}$ occurs $p$ times as a composition factor of ${\mathbb{F}[G]}$. Using Proposition \ref{Ind_B^G(S_a) composition factors}, ${V_1}$ occurs as a composition factor of ${\Ind_{B}^G(S_a)}$ only when ${a = 0}$, for which it occurs once. Thus, we wish to determine for which ${0\leq d\leq p-2}$ do we have ${(p-2)(m+d)\equiv 0\Mod{p-1}}$. This is equivalent to ${d \equiv -m\Mod{p-1}}$ for ${0\leq d\leq p-2}$, and hence ${d = -m - (p-1)\floor*{\frac{-m}{p-1}}}$. Therefore we obtain
    \begin{align*}
        d_1
        &=
        p\left(\frac{m}{p+1} - \frac{2m}{p(p^2-1)}\right)
        -
        \frac{1}{p-1}\left(-m-(p-1)\floor*{\frac{-m}{p-1}}\right)\\
        &=
        \frac{m(p-2)}{p-1} + \frac{m}{p-1} - \ceil*{\frac{m}{p-1}}
        \\
        &= m - \ceil*{\frac{m}{p-1}}\\
        &= 1 + \sigma_{p-1}.
    \end{align*}
    Next, ${V_p}$ has dimension $p$ and is both simple and projective. Thus ${V_p}$ occurs $p$ times as a composition factor of ${\mathbb{F}[G]}$. Furthermore, just as for ${V_1}$, ${V_p}$ only occurs in ${\Ind_{B}^G(S_a)}$ as a composition factor when ${a=0}$, for which it occurs once. Therefore, by the same calculation above, ${d_p = d_1 = 1 + \sigma_{p-1}}$.\\
    \\
    Finally, we consider ${V_i}$ for ${2 \leq i\leq p-1}$. We first assume that ${i \neq (p\pm 1)/2}$. The projective modules with ${V_i}$ as a composition factor are ${P_{V_{p-1-i}}}$, ${P_{V_i}}$ and ${P_{V_{p+1-i}}}$ (note that these modules are distinct since ${i\neq (p\pm 1)/2}$). For ${P_{V_i}}$ it occurs twice, while it occurs once for ${P_{V_{p-1-i}}}$ and ${P_{V_{p+1-i}}}$. Hence ${V_i}$ occurs as a composition factor of ${\mathbb{F}[G]}$ a total of ${p-1-i + 2i + p+1-i = 2p}$ times. It's easily verified the same result holds when ${i = (p\pm 1)/2}$. Next, ${V_i}$ occurs as a composition factor of ${\Ind_B^G(S_a)}$ when ${a = i-1}$ or ${a = p-i}$, for which it occurs once each (if ${p-i = i-1}$, then it occurs twice for that single value of $a$). Similar to before,
    \begin{align*}
        (p-2)(m+d)\equiv i-1\Mod{p-1}\text{ if and only if }d \equiv -m-i+1\Mod{p-1}.\\
        (p-2)(m+d)\equiv p-i\Mod{p-1}\text{ if and only if }d \equiv -m-p+i\Mod{p-1}.
    \end{align*}
    Thus, we obtain
    \begin{align*}
        d_i
        &
        =
        2p
        \left(
          \frac{m}{p+1}
          -
          \frac{2m}{p(p^2-1)}
        \right)
        -
        \frac{1}{p-1}\left(
            -m-i+1
            -
            (p-1)\floor*{\frac{-m-i+1}{p-1}}
        \right)\\
        &
        \ \ \ -
        \frac{1}{p-1}\left(
            -m-p+i
            -
            (p-1)\floor*{\frac{-m-p+i}{p-1}}
        \right)\\
        &=  \frac{2m(p-2)}{p-1} + \frac{m+i-1}{p-1}
        + \frac{m+p-i}{p-1}
        -
        \ceil*{\frac{m+i-1}{p-1}}
        - \ceil*{\frac{m+p-i}{p-1}} \\
        &=
        1 + 2m -
        \ceil*{\frac{m+i-1}{p-1}}
        - \ceil*{\frac{m+p-i}{p-1}} \\
        &= 1 + \sigma_{i-1} + \sigma_{p-i}.
    \end{align*}
\end{proof}
We now give our second proof of Theorem \ref{composition factors of H^0} via an explicit composition series.
\begin{proof}[Second proof of Theorem \ref{composition factors of H^0}]
    Let
    $$
    {
        \mathcal{T} = \{(i,j) \in \mathbb{Z}\times \mathbb{Z}: 0\leq j\leq p-1,\ 0\leq i\leq m(p-2)-j\} \cup \{(0,j) \in \mathbb{Z}\times \mathbb{Z}: p\leq j\leq m(p-2)\}.
    }
    $$
    We consider the natural action of $G$ on the affine coordinate ring ${R = \mathbb{F}[x,y]/(xy^p-x^py-1)}$ given by:
    $$
    {
        x^iy^j \cdot
        \left(
            \begin{array}{ll}
                 \alpha&\beta  \\
                 \gamma&\delta 
            \end{array}
        \right)
        = (\alpha x + \beta y)^i (\gamma x + \delta y)^j.
    }
    $$
    Note that a version of the reduction formula given in Lemma \ref{reduction lemma} holds in $R$ by simply replacing ${\omega_{ij}}$ with ${x^iy^j}$. From this it follows that ${S = \mathrm{span}_{\mathbb{F}}\{x^iy^j: (i,j) \in \mathcal{T}\}}$ is a submodule of ${R}$. Furthermore, it follows from Proposition \ref{basis for holomorphic polydifferentials} and Lemma \ref{action of element on holomorphic polydifferentials} that the natural map ${H^0(C,\Omega_{C}^{\otimes m})\to S}$ given by ${\omega_{ij} \mapsto x^iy^j}$ is an isomorphism of ${\mathbb{F}[G]}$ modules. We will give a composition series of ${S \cong H^0(C,\Omega_{C}^{\otimes m})}$. For ${0\leq r\leq m(p-2)}$, it follows from Lemma \ref{reduction lemma} that the following subspace of ${S}$ is an ${\mathbb{F}[G]}$ submodule:
    \begin{align}
        \label{X_r submodule definition}
        X_r := \Span_{\mathbb{F}}\{x^i y^j: (i,j) \in \mathcal{T}\text{ and }0\leq i+j\leq r\}.
    \end{align}
    Furthermore, these submodules form a filtration of $S$:
    $$
    {
        0\subset X_0 \subset X_1 \subset ... \subset X_{m(p-2)} = S.
    }
    $$
    For ${p\leq r\leq m(p-2)}$, let ${1\leq t_r\leq p-1}$ such that ${r\equiv t_r\Mod{p-1}}$. Define the subspace of ${X_r}$:
    \begin{align}
        \label{Y_r definition}
        Y_r&:=X_{r-1} + \Span_{\mathbb{F}}\{x^r, x^{r-1}y,..., x^{r-t_r+1}y^{t_r-1}; y^r\}\subseteq X_r.
    \end{align}
    We will prove in Proposition \ref{Y_r tilde is simple} that in fact, ${Y_r}$ is an ${\mathbb{F}[G]}$-submodule of ${X_r}$ and ${Y_r/X_{r-1} \cong V_{t_r+1}}$. In Proposition \ref{X_r/Y_r = V_(p-t_r)}, we will prove that ${X_r/Y_r \cong V_{p-t_r}}$. It's clear that for ${0\leq r\leq p-1}$, we have ${X_r/X_{r-1} \cong V_{r+1}}$. It follows that
    \begin{align}
        \label{composition series for H0}
        0\subset X_0 \subset X_1 \subset ... \subset X_{p-1}\subset Y_p \subset X_p \subset Y_{p+1}\subset X_{p+1} \subset ... \subset Y_{m(p-2)} \subset X_{m(p-2)} = S
    \end{align}
    is a composition series for ${S\cong H^0(C,\Omega_{C}^{\otimes m})}$. It remains to count how many times each ${V_i}$ occurs as a quotient in the above series. Since ${m(p-2) \geq p-1}$ for ${m\geq 2}$ and ${p\geq 3}$, each ${V_i}$ for ${i=1,...,p}$ shows up exactly once at the start ${0\subset X_0 \subset ... \subset X_{p-1}}$ of (\ref{composition series for H0}). In the remaining part of (\ref{composition series for H0}), ${V_1,V_p}$ occur exactly when ${t_r = p-1}$. For ${2 \leq i\leq p-1}$ we have ${t_r + 1 = i}$ if and only if ${t_r = i-1}$ and ${p-t_r = i}$ if and only if ${t_r = p-i}$. In other words,
    \begin{itemize}
        \item The number of times both ${V_1,V_p}$ occur in (\ref{composition series for H0}) is one more than the number of integers in ${[p,m(p-2)]}$ congruent to ${p-1\Mod{p-1}}$.
        \item For ${2\leq i\leq p-1}$ the number of times ${V_i}$ occurs in (\ref{composition series for H0}) is one more than the sum of:
        \begin{enumerate}
            \item the number of integers in ${[p,m(p-2)]}$ congruent to ${i-1\Mod{p-1}}$ and
            \item the number of integers in ${[p,m(p-2)]}$ congruent to ${p-i\Mod{p-1}}$.
        \end{enumerate}
    \end{itemize}
    Thus, applying the following Lemma \ref{l_i counting lemma} finishes the proof of Theorem \ref{composition factors of H^0}.
\end{proof}
\begin{lem}
    \label{l_i counting lemma}
    Let ${1\leq t\leq p-1}$. Then ${\sigma_{t}}$ as defined in equation (\ref{sigma_t}) is equal to the number of integers in the interval ${[p,m(p-2)]}$ congruent to ${t\Mod{p-1}}$.
\end{lem}
\begin{proof}
    Note that the smallest integer ${\geq p}$ equivalent to ${t\Mod{p-1}}$ is ${p-1+t}$. First assume that ${m(p-2) \geq p-1+t}$. Then the number of integers between ${p\leq r\leq m(p-2)}$ equivalent to ${p-1+t}$ modulo ${p-1}$ is given by
    $$
    {
        1 + \floor*{\frac{m(p-2)-(p-1+t)}{p-1}} =
        1 + \floor*{m - 1 - \frac{m+t}{p-1}} = m - \ceil*{\frac{m+t}{p-1}}
        = \sigma_t.
    }
    $$
    We now verify that if ${m(p-2) < p-1+t}$, we have ${\sigma_t = 0}$. Note that this implies ${m(p-2) - (p-1+t) < 0}$, and hence by using the left-hand side of the above expression we see ${\sigma_t \leq 0}$. By utilising the right-hand expression for ${\sigma_t}$, we see that
    $$
    {
        \sigma_t=
        m - \ceil*{\frac{m+t}{p-1}}
        \geq
        m - \ceil*{\frac{m+p-1}{p-1}}
        =
        m - 1 - \ceil*{\frac{m}{p-1}}
        \geq 0.
    }
    $$
    Therefore in this case ${\sigma_t = 0}$.
\end{proof}
We will now prove Propositions \ref{Y_r tilde is simple} and \ref{X_r/Y_r = V_(p-t_r)}. We need the following congruence:
\begin{lem}
    \label{glaishers congruence}
    Let ${r\geq 1}$, let ${1\leq t_r\leq p-1}$ such that ${r \equiv t_r\Mod{p-1}}$, and let ${1\leq l\leq p-1}$. Then
    $$
    {
        \sum_{k\geq 0}{r \choose l + k(p-1)} \equiv {t_r \choose l} \Modwb{p},
    }
    $$
    where we take ${{n\choose m} = 0}$ whenever ${m > n}$.
\end{lem}
\begin{proof}
    This is equivalent to Glaisher's congruence: see \cite[\S 1]{zwsun}.
\end{proof}
For the remainder of this section, denote the two elements
$$
{
    g =
    \left(
        \begin{array}{ll}
             1&1  \\
             0&1 
        \end{array}
    \right),\ h
    =
    \left(
        \begin{array}{ll}
             1&0  \\
             1&1 
        \end{array}
    \right)
}
$$
which are generators of ${G = SL_2(\mathbb{F}_p)}$ (see \cite[\S 2, Thm. 8.4]{gorenstein}).
\begin{lem}
    \label{action of g on x^{r-j}y^j}
    Let ${j \in [0,p-1] \cup \{r\}}$ and consider ${x^{r-j}y^j \in X_r}$. We have
    $$
    {
        x^{r-j}y^j\cdot g
        =
        \begin{dcases}
            y^r&\text{if }j=r\\
            \sum_{i=0}^{t_r-j-1}{t_r-j \choose i}x^{r-i-j}y^{i+j} + y^r + u_j,\ u_j \in X_{r-1}&\text{if }0\leq j\leq t_r-1\\
            \sum_{i=0}^{p-1-j}{p-1-j+t_r \choose i}x^{r-i-j}y^{i+j} + v_j,\ v_j \in Y_r&\text{if }t_r\leq j\leq p-1
        \end{dcases}.
    }
    $$
\end{lem}
\begin{proof}
    The case of ${j=r}$ is straightforward by definition. For ${0\leq j\leq p-1}$, we have
    $$
    {
        x^{r-j}y^j\cdot g
        =
        \sum_{i=0}^{r-j}{r-j \choose i}x^{r-j-i}y^{i+j}.
    }
    $$
    By using the reduction formula (Lemma \ref{reduction lemma}), the above sum gives the following coefficients of our basis elements in ${X_r \setminus X_{r-1}}$, along with some ${u_j \in X_{r-1}}$:
    \begin{align}
        &\text{Coefficient of }x^r:&&\begin{dcases}1&\text{if }j=0\\0&\text{otherwise}\end{dcases}\label{coeff_xr}\\
        &1\leq i\leq j-1,\ \text{coefficient of }x^{r-i}y^i:&&\sum_{0\leq p-1-j+i+k(p-1) < r-j}{r-j\choose p-1-j+i+k(p-1)}\label{coeff_firstsum}\\
       &\text{Coefficient of }x^{r-j}y^j: &&\sum_{0 \leq k(p-1) < r-j}{r-j \choose k(p-1)}\label{coeff_secondsum}\\
       &j+1\leq i\leq p-1,\ \text{coefficient of }x^{r-i}y^i:&& \sum_{0 \leq i-j+k(p-1) < r-j}{r-j \choose i-j+k(p-1)}\label{coeff_thirdsum}\\
       &\text{Coefficient of }y^r:&&{r-j\choose r-j} = 1 \label{coeff_yr}.
    \end{align}
    Throughout this proof, we make repeated use of Lemma \ref{glaishers congruence}, and thus we begin by remarking
    \begin{equation}
        \label{r-j congruence}
        r-j
        \equiv
        \left.
        \begin{dcases}
            t_r&\text{if }j=0\\
            t_r-j&\text{if }1\leq j\leq t_r-1\\
            p-1 + t_r-j&\text{if }t_r\leq j\leq p-1
        \end{dcases}
        \right\}\Mod{p-1},
    \end{equation}
    where the right-hand side takes values in ${\{1,2,...,p-1\}}$.\\
    \\
    We start by considering the first sum (\ref{coeff_firstsum}), which occurs only when ${j\geq 1}$. Note that this implies \\${1\leq p-1-j+i\leq p-1}$. Furthermore, there exists a ${k \in \mathbb{N}}$ such that ${p-1-j+i+k(p-1) = r-j}$ if and only if ${i = t_r}$, and so (\ref{coeff_firstsum}) becomes
    $$
    {
        \sum_{k\geq 0}{r-j\choose p-1-j+i+k(p-1)}
        -
        \begin{dcases}
            1&\text{if }i=t_r\\
            0&\text{otherwise}
        \end{dcases}.
    }
    $$
    After applying Lemma \ref{glaishers congruence} and (\ref{r-j congruence}), this becomes
    $$
    {
        \left.
        \begin{dcases}
            {t_r-j\choose p-1-j+i}&\text{if }1\leq j\leq t_r-1\\
            {p-1+t_r-j\choose p-1-j+i}&\text{if }t_r\leq j\leq p-1
        \end{dcases}
        \right\}
        -
        \begin{dcases}
            1&\text{if }i=t_r\\
            0&\text{otherwise}
        \end{dcases}.
    }
    $$
    Note it is only possible to have ${i=t_r}$ when ${t_r+1\leq j\leq p-1}$, since we assumed ${1\leq i\leq j-1}$. Hence the above simplifies to
    $$
    {
        \begin{dcases}
            {t_r-j\choose p-1-j+i}&\text{if }1\leq j\leq t_r-1\\
            {p-1+t_r-j\choose p-1-j+i}&\text{if }t_r\leq j\leq p-1\text{ and }i\neq t_r\\
            {p-1+t_r-j\choose p-1-j+t_r}-1&\text{if }t_r+1\leq j\leq p-1\text{ and }i=t_r
        \end{dcases}.
    }
    $$
    Note that since ${1\leq i\leq j-1}$, we have ${p-1+i\leq t_r}$, and hence ${p-1-j+i > t_r-j}$. Hence the top-most case above simplifies to $0$, and we obtain
    $$
    {
        \begin{dcases}
            {p-1+t_r-j\choose p-1-j+i}&\text{if }t_r\leq j\leq p-1\text{ and }i\neq t_r\\
            0&\text{otherwise}
        \end{dcases}.
    }
    $$
    Now we consider the second sum (\ref{coeff_secondsum}), and assume ${j\geq 1}$ (${j=0}$ is already covered by (\ref{coeff_xr})). Note that there exists a ${k \in \mathbb{N}}$ such that ${k(p-1) = r-j}$ if and only if ${j=t_r}$. Hence
    $$
    {
        \sum_{0 \leq k(p-1) < r-j}{r-j \choose k(p-1)}
        =
        1 + \sum_{k\geq 0}{r-j \choose p-1 + k(p-1)}
        -
        \begin{dcases}
            {r-j\choose r-j}=1&\text{if }j=t_r\\
            0&\text{otherwise}
        \end{dcases}
    }
    $$
    which, after applying Lemma \ref{glaishers congruence} gives
    $$
    {
        1
        +
        \begin{dcases}
            {t_r-j\choose p-1}&\text{if }1\leq j\leq t_r-1\\
            {p-1+t_r-j \choose p-1}&\text{if }t_r\leq j\leq p-1
        \end{dcases}
        -
        \begin{dcases}
            1&\text{if }j=t_r\\
            0&\text{otherwise}
        \end{dcases}.
    }
    $$
    The above always simplifies to $1$, which can be easily seen by taking cases. Hence the second sum simplifies to $1$. We now consider the final sum (\ref{coeff_thirdsum}). Since ${j+1\leq i\leq p-1}$ and ${0\leq j\leq p-1}$, we have ${1\leq i-j\leq p-1}$. Also, there exists a ${k \in \mathbb{N}}$ such that ${i-j + k(p-1) = r-j}$ if and only if ${i=t_r}$. Hence (\ref{coeff_thirdsum}) becomes
    $$
    {
        \sum_{k\geq 0}{r-j \choose i-j+k(p-1)}
        -
        \begin{dcases}
            1&\text{if }i=t_r\\
            0&\text{otherwise}
        \end{dcases}.
    }
    $$
    Applying Lemma \ref{glaishers congruence} and (\ref{r-j congruence}), the above becomes
    $$
    {
        \left.
        \begin{dcases}
            {t_r\choose i}&\text{if }j=0\\
            {t_r-j\choose i-j}&\text{if }1\leq j\leq t_r-1\\
            {p-1+t_r-j\choose i-j}&\text{if }t_r\leq j\leq p-1
        \end{dcases}
        \right\}
        -
        \begin{dcases}
            1&\text{if }i=t_r\\
            0&\text{otherwise}
        \end{dcases}.
    }
    $$
    Since ${j+1\leq i\leq p-1}$, this simplifies to
    $$
    {
        \left.
        \begin{dcases}
            {t_r\choose i}&\text{if }j=0\text{ and }i\neq t_r\\
            {t_r\choose t_r}-1&\text{if }j=0\text{ and }i=t_r\\
            {t_r-j\choose i-j}&\text{if }1\leq j\leq t_r-1\text{ and }i\neq t_r\\
            {t_r-j\choose t_r-j}-1&\text{if }1\leq j\leq t_r-1\text{ and }i=t_r\\
            {p-1+t_r-j\choose i-j}&\text{if }t_r\leq j\leq p-1\\
        \end{dcases}
        \right\}
        =
        \begin{dcases}
            0&\text{if }i=t_r\\
            {t_r-j\choose i-j}&\text{if }0\leq j\leq t_r-1\text{ and }i\neq t_r\\
            {p-1+t_r-j\choose i-j}&\text{if }t_r\leq j\leq p-1
        \end{dcases}.
    }
    $$
    Adding (\ref{coeff_xr})-(\ref{coeff_yr}), for ${0\leq j\leq p-1}$ we obtain
    $$
    {
        x^{r-j}y^{j}\cdot g=
        \begin{dcases}
            x^r + \sum_{i=1}^{t_r-1}{t_r\choose i}x^{r-i}y^{i}+y^r&\text{if }j=0\\
            x^{r-j}y^j + \sum_{i=j+1}^{t_r-1}{t_r-j\choose i-j}x^{r-i}y^{i}+ y^r&\text{if }1\leq j\leq t_r-1\\
            \sum_{i=1}^{t_r-1}{p-1\choose p-1-t_r+i}x^{r-i}y^i + x^{r-t_r}y^{t_r} + \sum_{i=t_r+1}^{p-1}{p-1\choose i-t_r}x^{r-i}y^{i}+y^r&\text{if }j=t_r\\
            \sum_{i=1}^{t_r-1}{p-1+t_r-j\choose p-1-j+i}x^{r-i}y^{i}
            +
            x^{r-j}y^j
            +
            \sum_{i=j+1}^{p-1}{p-1+t_r-j\choose i-j}x^{r-i}y^i+y^r&\text{if }t_r+1\leq j\leq p-1
        \end{dcases}
    }
    $$
    plus some ${u_j \in X_{r-1}}$. We note that the top two and bottom two cases can be combined, and defining ${v_j := \sum_{i=1}^{t_r-1}{p-1+t_r-j\choose p-1-j+i}x^{r-i}y^{i} + y^r + u_j \in Y_r}$ for the bottom two cases gives us
    $$
    {
        x^{r-j}y^{j}\cdot g=
        \begin{dcases}
            \sum_{i=0}^{t_r-j-1}{t_r-j \choose i}x^{r-i-j}y^{i+j} + y^r + u_j,\ u_j \in X_{r-1}&\text{if }0\leq j\leq t_r-1\\
            \sum_{i=0}^{p-1-j}{p-1-j+t_r \choose i}x^{r-i-j}y^{i+j} + v_j,\ v_j \in Y_r&\text{if }t_r\leq j\leq p-1
        \end{dcases},
    }
    $$
    as claimed in Lemma \ref{action of g on x^{r-j}y^j}.
\end{proof}
To state Propositions \ref{Y_r tilde is simple} and \ref{X_r/Y_r = V_(p-t_r)}, we will use the explicit description of ${V_t}$ as given in Proposition \ref{simple F[G] modules}.
\begin{prop}
    \label{Y_r tilde is simple}
    ${Y_r}$ is an ${\mathbb{F}[G]}$ submodule of ${X_r}$, and ${Y_r/X_{r-1} = \Span_{\mathbb{F}}\{\tilde{x}^r,...,\tilde{x}^{r-t_r+1}\tilde{y}^{t_r-1};\tilde{y}^r\} \cong V_{t_r + 1}}$, under the linear map ${\phi_r : Y_r/X_{r-1} \to V_{t_r+1}}$ determined by
    $$
    {
        \phi_{r}(\tilde{x}^{r-j}\tilde{y}^j)
        =
        \left\{
            \begin{array}{ll}
                 \mathcal{X}^{t_r-j}\mathcal{Y}^j&\text{if }j\neq r  \\
                 \mathcal{Y}^{t_r}&\text{if }j=r 
            \end{array}
        \right. .
    }
    $$
\end{prop}
\begin{proof}
    The map ${\phi_r}$ is clearly bijective. It remains to show that ${Y_r}$ is a submodule of ${X_r}$ and that ${\phi_r}$ is a $G$-map. For ${0\leq j\leq t_r-1}$, it's clear that ${x^{r-j}y^j\cdot h \in Y_r}$ and ${\phi_r(\tilde{x}^{r-j}\tilde{y}^j\cdot h) = \phi_r(\tilde{x}^{r-j}\tilde{y}^j)\cdot h}$. Furthermore, because ${y^r\cdot h = x^r\cdot g}$ and ${y^{t_r}\cdot h = x^{t_r}\cdot g}$, it suffices to verify that ${Y_r}$ is closed under the action of $g$, and that ${\phi_r}$ preserves the action of $g$. By Lemma \ref{action of g on x^{r-j}y^j}, when ${0\leq j\leq t_r-1}$ we have
    $$
    {
        x^{r-j}y^j\cdot g=
        \sum_{i=0}^{t_r-j-1}{t_r-j \choose i}x^{r-i-j}y^{i+j} + y^r + u_j,\ u_j \in X_{r-1},
    }
    $$
    from which we see that ${x^{r-j}y^j\cdot g \in Y_r}$. Passing to the quotient ${Y_r/X_{r-1}}$ and applying ${\phi_r}$ then gives
    $$
    {
        \phi_r(\tilde{x}^{r-j}\tilde{y}^j\cdot g)
        =
        \sum_{i=0}^{t_r-j}{t_r-j\choose i}\mathcal{X}^{t_r-j-i}\mathcal{Y}^{j+i}
        =
        \phi_{r}(\tilde{x}^{r-j}\tilde{y}^j)\cdot g.
    }
    $$
    It remains to consider the final basis element ${y^r}$. We have ${y^r\cdot g = y^r}$, and ${\phi_r(\tilde{y}^r\cdot g) = \mathcal{Y}^{t_r} = \mathcal{Y}^{t_r}\cdot g}$. We conclude ${Y_r}$ is a submodule of ${X_r}$, that ${\phi_r}$ is a ${G}$-map, and that ${Y_r/X_{r-1} \cong V_{t_r+1}}$.
\end{proof}
We will now also show that ${X_r/Y_r}$ is isomorphic to ${V_{p-t_r}}$. For this, we need the following:
\begin{lem}
    \label{j choose tr lemma}
    Let ${1\leq t\leq p-1}$, and let ${t\leq j\leq p-1}$. Then
    $$
    {
        \frac{{p-1-t \choose j-t}}{{p-1 \choose j-t}} = {j \choose t}\text{ in }\mathbb{F}_p.
    }
    $$
\end{lem}
\begin{proof}
    Note that
    $$
    {
        \frac{{p-1-t \choose j-t}}{{p-1 \choose j-t}}
        =
        \frac{(p-1-t)!(p-1-j+t)!}{(p-1)!(p-1-j)!}
        =
        \frac{\prod_{i=1}^{t}(p-1-j+i)}{\prod_{i=1}^{t}(p-1-t+i)}.
    }
    $$
    Reducing the numerator and denominator modulo $p$ we obtain
    $$
    {
        \frac{\prod_{i=1}^{t}(-1-j+i)}{\prod_{i=1}^{t}(-1-t+i)}
        =
        \frac{\prod_{i=1}^{t}(j-i+1)}{\prod_{i=1}^{t}(t-i+1)}
        =
        \frac{j!}{t!(j-t)!}
        =
        {j \choose t}.
    }
    $$
\end{proof}
\begin{prop}
    \label{X_r/Y_r = V_(p-t_r)}
    For ${p\leq r\leq m(p-2)}$, we have that ${X_r/Y_r = \Span_{\mathbb{F}}\{\tilde{x}^{r-t_r}\tilde{y}^{t_r},...,\tilde{x}^{r-(p-1)}\tilde{y}^{p-1}}\}$ is ${\mathbb{F}[G]}$-isomorphic to ${V_{p-t_r} = \Span_{\mathbb{F}}\{x^{p-t_r-1},...,y^{p-t_r-1}}\}$ under the bijective ${\mathbb{F}}$-linear map determined by
    $$
    {
        \lambda_r(\tilde{x}^{r-j}\tilde{y}^j)
        =
        {j \choose t_r}\mathcal{X}^{p-1-j}\mathcal{Y}^{j-t_r}\text{ for }j\in \{t_r,...,p-1\}.
    }
    $$
\end{prop}
\begin{proof}
    We first consider the action of ${h}$. Note that
    $$
    {
        \tilde{x}^{r-j}\tilde{y}^j\cdot h =
        \sum_{i=0}^{j}{j \choose i}\tilde{x}^{r-j+i}\tilde{y}^{j-i} = \sum_{i=0}^{j-t_r}{j \choose i}\tilde{x}^{r-j+i}\tilde{y}^{j-i}
    }
    $$
    where for the latter equality we have used the fact that in the quotient module ${X_r/Y_r}$, if ${j-i < t_r}$ then ${\tilde{x}^{r-j+i}\tilde{y}^{j-i} = 0}$. Applying ${\lambda_r}$, we obtain
    $$
    {
        \lambda_r(\tilde{x}^{r-j}\tilde{y}^j\cdot h) =
        \sum_{i=0}^{j-t_r}{j \choose i}{j-i \choose t_r}\mathcal{X}^{p-1-j+i}\mathcal{Y}^{j-i-t_r}.
    }
    $$
    By using the fact that ${{j \choose i}{j-i \choose t_r} = {j\choose t_r}{j-t_r \choose i}}$, we obtain
    $$
    {
        \lambda_r(\tilde{x}^{r-j}\tilde{y}^j\cdot h) =
        {j \choose t_r}\sum_{i=0}^{j-t_r}{j-t_r\choose i}\mathcal{X}^{p-1-j+i}\mathcal{Y}^{j-t_r-i} =
        {j \choose t_r} \mathcal{X}^{p-1-j}\mathcal{Y}^{j-t_r}\cdot h
        =
        \lambda_r({\tilde{x}^{r-j}\tilde{y}^j})\cdot h.
    }
    $$
    We now consider the action of ${g}$. By Lemma \ref{action of g on x^{r-j}y^j}, we have
    $$
    {
        \tilde{x}^{r-j}\tilde{y}^j \cdot g =
        \sum_{i=0}^{p-1-j}{p-1-j+t_r \choose i}\tilde{x}^{r-i-j}\tilde{y}^{i+j}.
    }
    $$
    Applying ${\lambda_r}$, we get
    $$
    {
        \lambda_r(\tilde{x}^{r-j}\tilde{y}^j\cdot g)
        =
        \sum_{i=0}^{p-1-j}
          {p-1-j+t_r \choose i}
          {i+j \choose t_r}
          \mathcal{X}^{p-1-j-i}\mathcal{Y}^{j-t_r+i}
    }
    $$
    which, after applying Lemma \ref{j choose tr lemma}, can be written as
    $$
    {
        =
        \sum_{i=0}^{p-1-j}{p-1-j+t_r \choose i}
            \frac{{p-1-t_r \choose i+j-t_r}}{{p-1 \choose i+j-t_r}}
          \mathcal{X}^{p-1-j-i}\mathcal{Y}^{j-t_r+i}.
    }
    $$
    Since ${{p-1-j+t_r \choose i}\frac{{p-1-t_r \choose i+j-t_r}}{{p-1 \choose i+j-t_r}} = {p-1-j \choose i}
    \frac{{p-1-t_r \choose j-t_r}}{{p-1 \choose j-t_r}}}$, we get
    $$
    {
        \lambda_r(\tilde{x}^{r-j}\tilde{y}^j\cdot g)
        =
        \frac{{p-1-t_r \choose j-t_r}}{{p-1 \choose j-t_r}}
        \sum_{i=0}^{p-1-j}{p-1-j \choose i}
          \mathcal{X}^{p-1-j-i}\mathcal{Y}^{j-t_r+i}
        =
        \frac{{p-1-t_r \choose j-t_r}}{{p-1 \choose j-t_r}} \mathcal{X}^{p-1-j}\mathcal{Y}^{j-t_r}\cdot g
        = \lambda_r(\tilde{x}^{r-j}\tilde{y}^{j})\cdot g
    }
    $$
    where the final equality comes from Lemma \ref{j choose tr lemma}. This completes the proof that ${\lambda_r}$ is a ${G}$-map.
\end{proof}

\section{Appendix}
In this appendix, we will compute the \textit{p-rank} and \textit{a-number} of the Drinfeld curve for arbitrary ${q=p^r}$. As explained in the introduction of \cite{Pries_2007}, these two important non-negative integral invariants are related to the $p$-torsion of the Jacobian of the curve. The first author would like to thank the anonymous referee of \cite{lucas} for raising the problem of computing the $a$-number. We start by giving a definition of the \textit{Cartier operator}.
\begin{defn}
    \label{cartier operator definition}
    Let $K$ be a function field in one variable defined over a perfect field $k$ of positive characteristic $p$. Let $x$ be a separating variable. Each function ${z \in K}$ can be written in a unique way as
    $$
    {
        z = z_0^p + z_1^px + ... + z_{p-1}^{p}x^{p-1},
    }
    $$
    where ${z_0,...,z_{p-1} \in K}$ (see the exercises at the end of \cite[\S 4]{stichtenoth2009}). The Cartier operator on differential forms, which we denote in this section by ${\mathcal{C}}$, is then defined by
    $$
    {
        \mathcal{C}(zdx) = z_{p-1}dx.
    }
    $$
\end{defn}
It is explained at the end of \cite[\S 4]{stichtenoth2009} that the above definition of ${\mathcal{C}}$ is independent from the choice of separating variable. Furthermore, ${\mathcal{C}}$ sends globally holomorphic differentials of the corresponding curve to globally holomorphic differentials, and is ${(1/p)}$-linear: ${\mathcal{C}(z^p\omega) = z\mathcal{C}(\omega)}$. We can now give a definition of the ${p}$-rank:
\begin{defn}
    \label{p-rank definition}
    The ${p}$-rank of $C$ is ${\dim_{\mathbb{F}}\{\omega \in H^0(C,\Omega_C): \mathcal{C}(\omega) = \omega\}}$, where ${\mathcal{C}}$ is the Cartier operator associated with the function field of $C$. A curve with a ${p}$-rank of $0$ is called \textit{supersingular}.
\end{defn}
\begin{prop}
    \label{p-rank of drinfeld curve}
    The Drinfeld curve $C$ is supersingular.
\end{prop}
\begin{proof}
    The Deuring–Shafarevich formula (for example, see \cite{ds_formula_resource}) states that if ${\mathcal{U}}$ is a $p$-subgroup of ${\Aut(C)}$, then
    $$
    {
        \gamma - 1 = |\mathcal{U}|(\overline{\gamma}-1) + \sum_{P \in C}\left(|\mathcal{U}|-\left|\frac{\mathcal{U}}{\mathcal{U}_P}\right|\right)
    }
    $$
    where: ${\gamma}$ is the $p$-rank of $C$ and ${\overline{\gamma}}$ is the $p$-rank of the quotient curve ${C/\mathcal{U}}$. Let ${\mathcal{U} = U}$ to be the subgroup of $G$ of upper uni-triangular matrices. By Proposition \ref{the quotient C/U}, we have ${C/U \cong \mathbb{P}^1(\mathbb{F})}$, which has genus $0$. Hence, the $p$-rank of ${\mathbb{P}^1(\mathbb{F})}$ is $0$. By Proposition \ref{the quotient C/U}, the only point ${P \in C}$ with ${U_P\neq 1}$ is ${P = [1:0:0]}$, with ${U_P = U}$. Thus, we get ${\gamma - 1 = q(-1) + (q-1) = -1}$, and hence ${\gamma = 0}$.
\end{proof}
We now go on to define the ${a}$-rank, and compute it for the Drinfeld curve.
\begin{defn}
    \label{a-number definition}
    The $a$-number of $C$ is ${\dim_{\mathbb{F}}\{\omega \in H^0(C,\Omega_C): \mathcal{C}(\omega) = 0\}}$.
\end{defn}
For ${i\in \mathbb{Z}}$, let ${0\leq r_i \leq p-1}$ such that ${i\equiv r_i \Modwb{p}}$. Then, define ${\phi : \mathbb{Z}^2 \to \mathbb{Z}^2}$ by
\begin{equation}
    \label{phi definition for a-number}
    \phi(i,j)
    =
    \left(
    \phi(i,j)_1,\ \phi(i,j)_2
    \right)
    =
    \left(
        \frac{i-r_i+q(p-1-r_j)}{p},\ 
        \frac{j-r_j+q(p-1-r_i)}{p}
    \right).
\end{equation}

\begin{lem}
    \label{cartier operator prop}
    Consider ${\omega_{ij} \in H^0(C,\Omega_C)}$ (as in Proposition \ref{basis for holomorphic polydifferentials}). Let ${r_i,r_j}$ be defined as above. If ${r_i + r_j \leq p-2}$, then ${\mathcal{C}(\omega_{ij}) = 0}$. Otherwise,
    $$
    {
        \mathcal{C}(\omega_{ij})
        =
        (-1)^{r_i + r_j - p + 1} {r_j \choose {p-1-r_i}}
        \omega_{
            \phi(i,j)
        }
    }
    $$
    where ${\phi(i,j)}$ is defined as in (\ref{phi definition for a-number}); furthermore, ${\phi(i,j)_1,\phi(i,j)_2 \geq 0}$ and ${\phi(i,j)_1 + \phi(i,j)_2 \leq q-2}$ (in other words, ${\omega_{\phi(i,j)}}$ is a basis element as in Proposition \ref{basis for holomorphic polydifferentials}).
\end{lem}
\begin{proof}
    We can write
    $$
    {
        \omega_{ij}
        =
        \left(\frac{x^{\frac{i-r_i}{p}}y^{\frac{j-r_j}{p}}}{x^{q/p}}\right)^p x^{r_i}y^{r_j} dx.
    }
    $$
    We have that $x$ is a separating element. To see this, note that ${\mathbb{F}(C) = \text{Frac}(\mathbb{F}[x,y]/(xy^q-x^qy-1)) = \mathbb{F}(x,y)}$, and that ${y \in \mathbb{F}(C)}$ is separable over ${\mathbb{F}(x)}$. Using the fact ${\mathcal{C}}$ is ${(1/p)}$-linear, we deduce that
    $$
    {
        \mathcal{C}(\omega_{ij})
        =
        \left(\frac{x^{\frac{i-r_i}{p}}y^{\frac{j-r_i}{p}}}{x^{q/p}}\right) \mathcal{C}(x^{r_i}y^{r_j} dx).
    }
    $$
    It remains to compute ${\mathcal{C}(x^{r_i}y^{r_j} dx)}$. Recall that ${x,y}$ are related by ${xy^q - x^qy - 1 = 0}$. Rearranging, we get
    $$
    {
        y = \left(-\frac{1}{x^{q/p}}\right)^p + \left(\frac{y^{q/p}}{x^{q/p}}\right)^p x.
    }
    $$
    If ${r_i + r_j \leq p-2}$, then the coefficient of ${x^{p-1}}$ in the expansion of ${x^{r_i}y^{r_j}}$ is $0$. Thus ${\mathcal{C}(x^{r_i}y^{r_j}dx) = 0}$ in this case, and so ${\mathcal{C}(\omega_{ij}) = 0}$. Otherwise, we obtain
    $$
    {
        y^{r_j}
        =
        \sum_{k=0}^{r_j}
            {r_j \choose k}
            \left((-1)^{r_j-k}\frac{1}{x^{q(r_j-k)/p}}\right)^{p}
            \left(\frac{y^{qk/p}}{x^{qk/p}}\right)^{p}
            x^{k}
        =
        \sum_{k=0}^{r_j}
            {r_j \choose k}
            \left(
                (-1)^{r_j-k}
                \frac{y^{qk/p}}{x^{qr_j/p}}
            \right)^{p}
            x^{k}
    }
    $$
    hence
    $$
    {
        x^{r_i}y^{r_j}
        =
        \sum_{k=0}^{r_j}
            {r_j \choose k}
            \left(
                (-1)^{r_j-k}
                \frac{y^{qk/p}}{x^{qr_j/p}}
            \right)^{p}
            x^{k+r_i}.
    }
    $$
    Note that ${x^{k + r_i}}$ contributes to ${\mathcal{C}(\omega_{ij})}$ if and only if ${k + r_i \equiv (p-1)\Modwb{p}}$. Since ${r_i + r_j \leq 2p-2}$, this happens only when ${k + r_i = p-1}$, i.e. ${k = p-1-r_i}$. Thus, we overall deduce that
    $$
    {
        \mathcal{C}(x^{r_i}y^{r_j}dx)
        =
        {r_j \choose p-1-r_i}
            \left(
                (-1)^{r_i + r_j-p+1}
                \frac{y^{q(p-1-r_i)/p}}{x^{qr_j/p}}
            \right)
        dx.
    }
    $$
    Finally, we get
    \begin{align*}
      \mathcal{C}(\omega_{ij})
        &=
        \left(\frac{x^{\frac{i-r_i}{p}}y^{\frac{j-r_i}{p}}}{x^{q/p}}\right) \mathcal{C}(x^{r_i}y^{r_j} dx)
        =
        (-1)^{r_i + r_j-p+1}
        {r_j \choose p-1-r_i}
        \left(\frac{x^{\frac{i-r_i}{p}}y^{\frac{j-r_j}{p}}}{x^{q/p}}\right)
        \left(
                \frac{y^{q(p-1-r_i)/p}}{x^{qr_j/p}}
        \right)
        dx\\
        &= (-1)^{r_i + r_j-p+1}
        {r_j \choose p-1-r_i}
        \left(
                \frac{x^{\frac{i-r_i + q(p-1-r_j)}{p}}y^{\frac{j-r_j+q(p-1-r_i)}{p}}}{x^{q}}
        \right) \\
        &= (-1)^{r_i + r_j-p+1}
        {r_j \choose p-1-r_i}\omega_{\phi(i,j)}.
    \end{align*}
    It's clear that ${\phi(i,j)_1,\phi(i,j)_2\geq 0}$. It remains to show ${\phi(i,j)_1 + \phi(i,j)_2 \leq q-2}$. For this, first note that this is equivalent to showing
    $$
    {
        (i + j) - (r_i + r_j)
        + q(2(p-1) - (r_i + r_j))
        \leq qp-2p.
    }
    $$
    To show the above inequality is true, we will consider the largest possible value of ${i+j}$ given a value of ${r_i + r_j}$ and show that the inequality holds in those cases. Recall that ${p-1\leq r_i + r_j\leq 2p-2}$ by assumption. The largest value of ${i+j}$ when ${r_i + r_j = 2p-2}$ is ${i + j = q-2}$. The largest value of ${i+j}$ when ${r_i + r_j = 2p-3}$ is thus ${i + j = q-3}$. In general, for ${0\leq k\leq p-1}$, the largest value of ${i+j}$ when ${r_i + r_j = 2p - 2 - k}$ is ${i + j = q - 2 - k}$. Thus, for a given ${r_i + r_j = 2p-2-k}$, the left-hand side of our inequality is bounded by
    \begin{align*}
        &(i + j) - (r_i + r_j)
        + q(2(p-1) - (r_i + r_j))\\
        \leq&(q-2-k) - (2p-2-k) + q(2(p-1) - (2p-2-k))\\
        =&q-2-2p+2 + qk\\
        =&q(k+1)-2p\\
        \leq& qp-2p.
    \end{align*}
\end{proof}

\begin{thm}
    \label{a-number computation}
    The $a$-number of the Drinfeld curve $C$ is ${\frac{q(q+p)(p-1)}{4p}}$.
\end{thm}
Note that if ${q=p}$, we obtain ${{p \choose 2} = g(C)}$.
\begin{proof}
    Let ${\mathcal{W} := \{(i,j): i,j\geq 0,\ 0\leq i+j\leq q-2\}}$. Our first step is to show that ${\phi|_{\mathcal{W}} : \mathcal{S} \to \mathbb{Z}^2}$ is injective. Assume that ${\phi(i,j) = \phi(k,w)}$. Write ${i = m_ip + r_i}$ and ${k = m_kp + r_k}$. Note then that
    $$
    {
        0
        =f
        \phi(i,j)_1 - \phi(k,w)_1
        =
        \frac{m_ip - m_kp + q(r_w - r_j)}{p}
    }
    $$
    hence
    $$
    {
        m_ip - m_kp = q(r_j - r_w).
    }
    $$
    Since both ${m_ip,m_kp}$ are less than $q$ (as ${i,k \leq q-2}$), we have ${-q < m_ip - m_kp < q}$. Thus, we deduce ${r_j = r_w}$ and ${m_ip = m_kp}$. Applying the same argument to the pair ${(j,w)}$, we deduce that ${r_i = r_k}$ and ${m_jp = m_wp}$, and hence ${i=k}$ and ${j = w}$. Overall, ${\phi|_{\mathcal{W}}}$ is injective.\\
    \\
    From Lemma \ref{cartier operator prop} and the injectivity of ${\Phi|_\mathcal{W}}$ we conclude that ${\ker(\mathcal{C})}$ is spanned by \\${\{\omega_{ij}: (i,j) \in \mathcal{W},\ r_i + r_j\leq p-2\}}$. Hence, the $a$-number is equal to
    $$
    {
        \#
        \{
            (i,j) \in \mathbb{Z}^2:\ 
            0\leq i,j\leq q-2,\ 0\leq i+j\leq q-2,\ r_i + r_j \leq p-2
        \}.
    }
    $$
    This can be computed as the sum:
    \begin{equation}
        \label{sum for computing the order of S1}
        \sum_{r_i=0}^{p-2}
        \sum_{r_j=0}^{p-2-r_i}
        \sum_{n_i=0}^{\floor*{(q-2-r_j-r_i)/p}}
        \sum_{n_j=0}^{\floor*{(q-2-r_j-r_i-n_ip)/p}}
        1.
    \end{equation}
    To see this, write ${i = n_ip + r_i}$ and ${j = n_jp + r_j}$. For a given choice of ${0\leq r_i\leq p-2}$, there are ${p-2-r_i}$ choices for ${r_j}$. Then, since we require ${i + j \leq q-2}$, there are ${\floor*{\frac{q-2-r_j-r_i}{p}}}$ choices for ${n_i}$. After choosing ${n_i}$, using ${i+j\leq q-2}$ once again there are ${\floor*{\frac{q-2-r_j-r_i-n_ip}{p}}}$ choices for ${n_j}$.\\
    \\
    We now simplify the sum. First of all, we simplify the bounds of the inner two sums. Note that because of the ranges of ${r_i,r_j}$, ${\ceil*{\frac{r_i+r_j+2}{p}} = 1}$. Hence,
    $$
    {
        \floor*{\frac{q-2-r_j-r_i}{p}} =
        \frac{q}{p} - \ceil*{\frac{r_i+r_j+2}{p}} = \frac{q}{p} - 1.
    }
    $$
    Also,
    $$
    {
        \floor*{\frac{q-2-r_j-r_i - n_ip}{p}}
        =
        \frac{q}{p} - n_i - \ceil*{\frac{r_i+r_j+2}{p}} = \frac{q}{p} - n_i - 1.
    }
    $$
    Hence, (\ref{sum for computing the order of S1}) becomes
    \begin{align*}
        &\sum_{r_i=0}^{p-2}
        \sum_{r_j=0}^{p-2-r_i}
        \sum_{n_i=0}^{q/p - 1}
        \sum_{n_j=0}^{q/p - n_i - 1}
        1\\
        =&\left(
            \sum_{r_i=0}^{p-2}
            \sum_{r_j=0}^{p-2-r_i}1
        \right)
        \left(
            \sum_{n_i=0}^{q/p - 1}
            \sum_{n_j=0}^{q/p - n_i - 1}1
        \right)\\
        =&\left(
            \sum_{r_i=0}^{p-2}
            (p-1-r_i)
        \right)
        \left(
            \sum_{n_i=0}^{q/p - 1}
            \left(\frac{q}{p}-n_i\right)
        \right)\\
        =&\frac{1}{2}p\left(p-1\right)\cdot \frac{1}{2}\frac{q}{p}\left(\frac{q}{p}+1\right)\\
        =&\frac{q(q+p)(p-1)}{4p}.
    \end{align*}
\end{proof}

\section{Bibliography}
\printbibliography[heading=none]

\end{document}